\documentclass{amsart}
\usepackage{amsmath}
\usepackage{amssymb}
\usepackage{amsthm}
\usepackage{enumerate}
\usepackage[pdftex]{graphicx}
\usepackage{caption}
\usepackage{amscd}
\theoremstyle{definition}
\newtheorem{definition}{Definition}[section]
\theoremstyle{plain}
\newtheorem{lemma}[definition]{Lemma}
\newtheorem{theorem}[definition]{Theorem}

\newtheorem{proposition}[definition]{Proposition}
\newtheorem{corollary}[definition]{Corollary}
\theoremstyle{remark}
\newtheorem{remark}[definition]{Remark}

\newtheorem{example}[definition]{Example}

\makeatletter
\@namedef{subjclassname@2020}{
  \textup{2020} Mathematics Subject Classification}
\makeatother
\scrollmode
\newcommand{\mycl}{\operatorname{cl}}
\newcommand{\myint}{\operatorname{int}}

\newcommand{\myrank}{\operatorname{rank}}
\newcommand{\myLim}{\operatorname{Lim}}
\newcommand{\myIso}{\operatorname{iso}}
\newcommand{\myLpt}{\operatorname{lpt}}
\newcommand{\myedim}{\operatorname{eRank}}

\begin{document}
\title[Definable quotients in d-minimal structures]{Definable quotients in d-minimal structures}
\author[M. Fujita]{Masato Fujita}
\address{Department of Liberal Arts,
Japan Coast Guard Academy,
5-1 Wakaba-cho, Kure, Hiroshima 737-8512, Japan}
\email{fujita.masato.p34@kyoto-u.jp}

\begin{abstract}
We consider d-minimal expansions of ordered fields.
We demonstrate the existence of definable quotients of definable sets by definable equivalence relations when several technical conditions are satisfied. These conditions are satisfied when there is a definable proper action of a definable group $G$ on a locally closed definable subset $X$ of $F^n$, where $F$ is the universe.
\end{abstract}



\maketitle

\section{Introduction}\label{sec:intro}
Several structures relaxing the constraints of o-minimal structures \cite{D} are proposed, for instance, weakly o-minimal structures \cite{MMS} and locally o-minimal structures \cite{TV}.
D-minimal structures are ones of them.
Miller proposed d-minimality in \cite{Miller-dmin}, but his interest is mainly in the case where the underlying space is the set of reals.
Fornasiero gave another definition of d-minimality applicable in more general setting.
We adopt Fornasiero's definition in this paper.

\begin{definition}
	An expansion $\mathcal F=(F,<,\ldots)$ of a dense linear order without endpoints
	is \textit{definably complete} if every definable subset of $F$ has both a supremum and an infimum in 
	$F \cup \{ \pm \infty\}$ \cite{M}.
	
	The structure $\mathcal F$  is \textit{d-minimal} if it is definably complete, and every definable subset $X$ of $F$ is the union of an open set and finitely many discrete sets, where the number of discrete sets does not depend on the parameters of definition of $X$	
	\cite{Fornasiero}.
\end{definition}

We demonstrate the existence of definable quotients in d-minimal structures.
For instance, the following theorem holds:
\begin{theorem}
	Consider a d-minimal expansion of an ordered field.
	Let $G$ be a definable group and $X$ be a definable $G$-set which is locally closed.
	Assume further that the $G$-action on $X$ is definably proper.
	Then, there exists a definable quotient $X \to X/G$.
\end{theorem}
\begin{proof}
	The theorem follows from Theorem \ref{thm:definable_action2} and Theorem \ref{thm:summary}. 
\end{proof}

We recall previous studies on definable quotients.
In real algebraic geometry, quotients of semialgebraic spaces are studied by Brumfiel \cite{B} and Scheiderer \cite{Sch}.
In o-minimal setting, definable quotients are introduced in \cite[Chapter 10, Section 2]{D}.
We have already demonstrated that almost the same assertions in this paper hold when the structure is a definably complete locally o-minimal expansion of an ordered field in the previous paper \cite{FK}.
The previous paper contains new results even when the structure is o-minimal, but many of them are already known in the o-minimal case such as in \cite[Chapter 10, Section 2]{D}.
Therefore, many results in the previous paper are not new when the universe is the set of reals, which is often assumed in applications, because locally o-minimal expansions of the ordered real field are o-minimal.
On the other hand, non-o-minimal examples of d-minimal expansion of the ordered real field are known \cite{D2, FM, MT}.
This is the reason why we treat d-minimal expansions of ordered fields in this paper.

We develop several new technical notions which are not used in the previous paper \cite{FK}.
We give equivalence conditions for definable maps to be definably proper and definably identifying in both the previous paper and this paper.
These conditions are described in term of definable curves in the previous paper which is also used in the o-minimal case \cite{D}.
However, in the d-minimal case, we have to develop a new notion called \textit{$(0+)$-sets} instead of definable curves.
The definition of {$(0+)$-sets} is given in Section \ref{sec:d-minimal}.
We also need to develop the notion of \textit{extended rank} though the traditional dimension function \cite{Fuji-dimension} meets our needs in the previous paper.
The main contribution of this paper is the development of these notions.

The whole copy of the proofs in the previous paper \cite{FK} does not work in the d-minimal case because we introduce new notions described above.
In fact, we need to modify the proofs here and there.
However, the modifications are all technical, and the strategy employed for the proofs in the d-minimal case is the same as that in the previous paper.
We expect that the same proof works for some unknown structure satisfying several technical conditions though we failed to find structures which satisfy the above conditions other than definably complete locally o-minimal expansions of ordered fields and d-minimal expansions of ordered fields. 
We clarify the sufficient conditions for structures to admit definable quotients.
The conditions are as follows:
\begin{itemize}
\item The structure enjoys the definable choice property;
\item It admits a good $(0+)$-pseudo-filter $\mathfrak D$;
\item It has a good extended rank function $\myedim$.
\end{itemize}
The definitions of above notions are found in Section \ref{sec:preliminary} and Section \ref{sec:proper}.
We prove that structures satisfying the above conditions admit definable quotients in Section \ref{sec:preliminary} through Section \ref{sec:proper_action}.
We finally prove that a d-minimal expansion of an ordered field satisfies these conditions in Section \ref{sec:d-minimal}.
In fact, the set of $(0+)$-sets becomes a good $(0+)$-pseudo-filter and the extended rank satisfies the requirements of good extended rank functions.
We use definable Tietze extension theorem in the proof of Section \ref{sec:def_prop_quo}.
It is proved in \cite[Lemma 6.6]{AF} for a special case, and the details of the proof is omitted.
We generalize it in Appendix \ref{sec:extension_thm} with the additional assumption that the function is bounded.

We clarify our basic notations.
Let $\mathcal F=(F,<,\ldots)$ be an expansion of a dense linear order without endpoints.
An open interval is denoted by $(a,b):=\{x \in F\;|\;a<x<b\}$, where $a,b \in F \cup\{\pm \infty\}$. 
A closed interval is denoted by $[a,b]$ and we naturally interpret $(a,b]$ and $[a,b)$.
The term `definable' means `definable in the given structure with parameters' in this paper.
The set $F$ has the topology induced from the order $<$.
The Cartesian product $F^n$ equips the product topology.
The topology of a subset of $F^n$ is the relative topology.
When $\mathcal F$ is an expansion of an ordered group, for $x=(x_1, \dots, x_n) \in F^n$, we denote $|x|=\max_{1 \le i \le n}|x_i|$.
We also set $x-y:=(x_1-y_1,\ldots, x_n-y_n) \in F^n$ for $x=(x_1,\ldots,x_n), y=(y_1,\ldots, y_n) \in F^n$.
Let $X$ be a topological space and $A$ be a subset of $X$.
We denote the frontier, the interior and the closure of $A$ by $\partial_XA$, $\myint_X(A)$ and $\mycl_X(A)$, respectively.
We omit the subscript $X$ when it is clear from the context.
Consider a definable subset of $F^n$.
We say that it is open/closed/locally closed when it is open/closed/locally closed in $F^n$.
Otherwise, we clearly describe in which space it is open/closed/locally closed.

\section{Preliminary and Extended dimension}\label{sec:preliminary}

We first introduce the notion of definable choice property.
\begin{definition}
	Consider a structure.
	A \textit{definable equivalence relation} $E$ on a definable set $X$ is a definable subset of $X \times X$ such that the binary relation $\sim_E$ defined by $x \sim_E y \Leftrightarrow (x,y) \in E$ is an equivalence relation defined on $X$.
	
	A structure $\mathcal F=(F,\ldots)$ enjoys the \textit{definable choice property} if the following condition is satisfied:
	Let $\pi:F^{m+n} \rightarrow F^m$ be a coordinate projection.
	Let $X$ and $Y$ be definable subsets of $F^m$ and $F^{m+n}$, respectively,  satisfying the equality $\pi(Y)=X$.
	There exists a definable map $\varphi:X \rightarrow Y$ such that $\pi(\varphi(x))=x$ for all $x \in X$.
	Furthermore, if $E$ is a definable equivalence relation defined on a definable set $X$, there exists a definable subset $S$ of $X$ such that $S$ intersects at exactly one point with each equivalence class of $E$. 
\end{definition}
For instance, definably complete locally o-minimal expansions of ordered groups \cite[Lemma 2.8]{Fuji_decomp} and d-minimal expansions of ordered groups \cite{Miller-choice} enjoy the definable choice property.

We next recall several basic assertions on definably complete structures.
\begin{definition}
	A definable subset $X$ of $F^n$ is \textit{definably compact} if it is closed and bounded in $F^n$.
\end{definition}

\begin{lemma}\label{lem:image}
	Consider a definably complete structure.
	The image of a definably compact set under a definable continuous map is again definably compact.
\end{lemma}
\begin{proof}
	See \cite[Proposition 1.10]{M}.
\end{proof}

\begin{proposition}\label{prop:definably_compact}
	Consider a definably complete structure and a definably compact set $X$.
	Let $\{C_c\}_{c \in I}$ be the definable family of decreasing definable closed subsets of $X$, where $I$ is an open interval.
	In other word, the set $\{(x,c) \in X \times I\;|\; x \in C_c\}$ is definable.
	Then the intersection $\bigcap_{c \in I}C_c$ is not empty.
\end{proposition}
\begin{proof}
	See \cite[Section 8.4]{J} and \cite[Remark 5.6]{FKK}.
\end{proof}

\begin{theorem}[Definable Tietze extension theorem for bounded functions]\label{thm:tietze}
	Consider a definably complete expansion of an ordered field $\mathcal F=(F,<, +,\cdot,0,1,\ldots)$.
	Let $A \subseteq X$ be definable subsets of $F^m$ such that $A$ is closed in $X$.
	Any bounded definable continuous function $\varphi:A \rightarrow F$ has a definable continuous extension $\Phi:X \rightarrow F$.
\end{theorem}
\begin{proof}
	The space $F^m$ is a definable metric space in the canonical way and $X$ is so.
	The theorem holds without the assumption that $f$ is bounded by \cite[Lemma 6.6]{AF} when $X=F^n$.
	The general case follows from Theorem \ref{thm:tietze_abstract}.
\end{proof}

We next introduce the notion of good extended rank function.

\begin{definition}\label{def:extended_rank}
	Let $\mathcal F=(F,<,\ldots)$ be an expansion of a dense linear order without endpoints.
	Let $\operatorname{Def}(F^n)$ be the set of definable subsets of $F^n$.
	The structure $\mathcal F$ has a \textit{good extended rank function} if, for each nonnegative integer $n$, there exists a map $\myedim_n: \operatorname{Def}(F^n) \to \mathcal E_n \cup \{-\infty\}$ satisfying the following conditions:
	\begin{enumerate}
		\item[(a)] The target space $\mathcal E_n$ is a linearly ordered set having a smallest element and $-\infty$ is smaller than any element in $\mathcal E_n$.
		\item[(b)] The equality $\myedim_n(X)=-\infty$ holds if and only if $X$ is an empty set.
		\item[(c)] There is no infinite decreasing sequence $e_1 > e_2 > \cdots$ of elements in $\mathcal E_n$.
		Here, an infinite decreasing sequence means a map $e:\mathbb N \to \mathcal E_n$ such that $e(m_1)>e(m_2)$ for each $m_1<m_2 \in \mathbb N$.
		\item[(d)] The equality $\myedim_n(A \cup B)=\max\{\myedim_n(A ),\myedim_n(B) \}$ holds;
		\item[(e)] Let $X, Y \in \operatorname{Def}(F^n)$ with $X \subseteq Y$.
		If $\myedim_n(X)=\myedim_n(Y)$, there exists a definable subset $U$ of $X$ such that $U$ is open in $Y$ and $\myedim_n(X \setminus U)<\myedim_n(Y)$. 
		\item[(f)] A nonempty subset $X$ of $F^n$ belonging to $\operatorname{Def}(F^n)$ is discrete whenever $\myedim_n(X)$ is the smallest element in $\mathcal E_n$.
	\end{enumerate}
	We simply denote $\myedim_n(X)$ by $\myedim(X)$ when $n$ is clear from the context.
\end{definition}

\begin{remark}
	Property (c) of Definition \ref{def:extended_rank} allows a special type of induction argument.
	Let $P$ be a property of definable sets.
	Every nonempty definable set has property $P$ if the following two conditions are satisfied:
	\begin{enumerate}
		\item[(1)] Property $P$ holds for every definable set $X$ whose extended rank $\myrank_n(X)$ is the smallest element in $\mathcal E_n$;
		\item[(2)] For any $d \in \mathcal E_n$ which is not the smallest in $\mathcal E_n$ and any definable set $X$ of $\myrank_n(X)=d$, either $X$ has property $P$ or there exists a definable set $Y$ with $\myrank_n(Y)<d$ such that $X$ enjoys property $P$ whenever $Y$ enjoys property $P$.
	\end{enumerate}
	If $X$ does not possess property $P$, we can construct infinitely many definable sets $Y_1,Y_2,\ldots$ satisfying $\myrank_n(Y_1)>\myrank_n(Y_2)>\ldots$. 
	It is a contradiction to property (c) of Definition \ref{def:extended_rank}.
\end{remark}

\begin{definition}[\cite{TV}]
	An expansion $\mathcal F=(F,<,\ldots)$ of a dense linear order without endpoints
	is \textit{locally o-minimal} if, for every definable subset $X$ of $F$ and for every point $a\in F$, there exists an open interval $I$ such that $a \in I$ and $X \cap I$ is a finite union of points and 
	open intervals.
\end{definition}

\begin{example}\label{ex:local1}
Consider a definably complete locally o-minimal structure $\mathcal F=(F,<,\ldots)$.
	We consider that $F^0$ is a singleton with the trivial topology.
Let $X$ be a nonempty definable subset of $F^n$.
The dimension of $X$ is the maximal nonnegative integer $d$ such that $\pi(X)$ has a nonempty interior for some coordinate projection $\pi:F^n \rightarrow F^d$.
We set $\dim(X)=-\infty$ when $X$ is an empty set.

The dimension function $\dim$ defined above possesses properties (a) through (f) in Definition \ref{def:extended_rank}.
\end{example}
\begin{proof}
	The dimension function $\dim$ obviously has properties (a) through (c).
	Property (d) follows from \cite[Proposition 2.8(5)]{FKK}.
	Property (e) is \cite[Lemma 2.16]{FK}.
	Property (f) follows from \cite[Proposition 2.8(1)]{FKK}.
\end{proof}

\begin{lemma}\label{lem:contained}
	Let $\mathcal F=(F,<,\ldots)$ be an expansion of a dense linear order without endpoints having a good extended rank function $\myedim$.
	Let $S \subseteq X$ be definable subsets of $F^n$.
	Then, the inequality $\myedim \partial_X S<\myedim S$ holds.
\end{lemma}
\begin{proof}
	Set $Y=\mycl_X(S)$ and $T=\partial_XS$.
	It is obvious that $\mycl_X(S)=\mycl_Y(S)$.
	Therefore, we may assume that $X=\mycl_X(S)$ without loss of generality.
	We have $X=S \cup T$ under this assumption.
	If $\myedim T = \myedim X$, there is a definable subset $V$ of $T$ such that it is open in $X$ by property (e) in Definition \ref{def:extended_rank}.
	It is a contradiction to the definition of $T$.
	We have $\myedim T<\myedim X$.
	On the other hand, we have $\myedim X  =\myedim S $ by property (d) in Definition \ref{def:extended_rank} because $X=S \cup T$ and $\myedim T <\myedim X$.
	We get the inequality $\myedim T  < \myedim S$.
\end{proof}

\begin{lemma}\label{lem:constructible}
	Let $\mathcal F=(F,<,\ldots)$ be an expansion of a dense linear order without endpoints having a good extended rank function $\myedim$.
	Definable sets are constructible; that is, it is the union of finitely many locally closed sets.
\end{lemma}
\begin{proof}
	Let $X$ be a definable subset of $F^n$.
	We prove that $X$ is constructible by induction on $d=\myedim_n X$.
	When $d$ is the smallest element in $\mathfrak E_n$, it is closed by Lemma \ref{lem:contained}.
	We have nothing to prove in this case.
	In the other case,  we have $\myedim_n(\partial _{F^n}X)<d$ by Lemma \ref{lem:contained}.
	The frontier $\partial _{F^n}X$ is constructible by the induction hypothesis.
	Since $X=\mycl_{F^n}(X) \setminus \partial _{F^n}X$, $X$ is also constructible.
\end{proof}

\section{Definably proper and definably identifying maps}\label{sec:proper}

We want to find equivalent conditions for definable maps to be definably proper and definably identifying.
We introduce the notion of good $(0+)$-pseudo-filters for that purpose.

\begin{definition}\label{def:filter}
	Consider an expansion of an ordered additive group $\mathcal F=(F,<,+,0,\ldots)$.
	Let $\mathfrak D$ be a family of definable subsets of $(0,\infty)$ whose closure in $F$ contains the origin $0$.
	For any $D \in \mathfrak D$ and a definable map $f: D \to F^n$, the \textit{limit set} $\myLim(f)$ of $f$ is the set of points $p \in F^n$ such that the intersection $f^{-1}(B) \cap I$ is not empty for any open box $B$ in $F^n$ containing the point $p$ and any open interval $I$ containing the point $0$.
	We say that a nonempty family $\mathfrak D$ of definable subsets of $(0,\infty)$ is a \textit{good $(0+)$-pseudo-filter} if the following conditions are satisfied:
	\begin{enumerate}
		\item[(a)] For each $D \in \mathfrak D$, the set $D$ is bounded in $F$ and the equality $\mycl_F(D) = D \cup \{0\}$ holds.
		\item[(b)] For each $D \in \mathfrak D$ and definable sets $D_1$ and $D_2$ with $D=D_1 \cup D_2$, either $D_1$ or $D_2$ contains an element of $\mathfrak D$.
		\item[(c)] Let $D \in \mathfrak D$ and $\gamma:D \to F^n$ be a definable map.
		There exists $D' \in \mathfrak D$ such that $D' \subseteq D$ and the restriction of $\gamma$ to $D'$ is continuous.
		Furthermore, when $\myLim(\gamma) \neq \emptyset$, we can take $D'$ so that $\myLim(\gamma|_{D'})$ is a singleton, where $\gamma|_{D'}$ is the restriction of $\gamma$ to $D'$.
	\end{enumerate}
	We say that the structure $\mathcal F$ admits a \textit{good $(0+)$-pseudo-filter} if there exists a good $(0+)$-pseudo-filter.
	
	When a good $(0+)$-pseudo-filter $\mathfrak D$ is given, a $\mathfrak D$-\textit{pseudo-curve} is a definable continuous map $f:D \to F^n$ such that $D \in \mathfrak D$.
	We simply call it a \textit{pseudo-curve} when $\mathfrak D$ is clear from the context. 
\end{definition}

\begin{example}\label{ex:local2}
Consider a definably complete locally o-minimal expansion of an ordered group $\mathcal F=(F,<,+,0,\ldots)$.
Let $\mathfrak D$ be the family of intervals of the form $(0,\varepsilon]$ with $\varepsilon >0$.
It is a good $(0+)$-pseudo-filter.
\end{example}
\begin{proof}
It is obvious that $\mathfrak D$ enjoys property (a) in Definition \ref{def:filter}.
Property (b) immediately follows from local o-minimality.
	We finally check that $\mathfrak D$ enjoys property (c).
	Let $\gamma:(0,\varepsilon]  \to  F^n$ be a definable map.
	We may assume that $\gamma$ is continuous by the local monotonicity theorem \cite[Theorem 2.3]{FKK} by taking smaller $\varepsilon >0$ if necessary.
	We consider the case in which $\myLim(\gamma)$ is not empty.
	The set $\myLim(\gamma)$ is contained in the frontier of $\gamma((0,\varepsilon])$.
	It implies that $\myLim(\gamma)$ is discrete and closed by \cite[Proposition 2.8(1)]{FKK}.
	Take a point $x \in \myLim(\gamma)$ and definable open neighborhood $U$ of $x$ so that $U \cap \myLim(\gamma)=\{x\}$.
	By local o-minimality, if we take smaller $\varepsilon>0$ again, we may assume that $\gamma((0,\varepsilon]) \subseteq U$.
	It implies $\myLim(\gamma)=\{x\}$.
\end{proof}

We next investigate basic properties of $\mathfrak D$.

\begin{lemma}\label{lem:basic_o+1}
	Consider a definably complete expansion $\mathcal F=(F,+,<,0,\ldots)$ of an ordered group.
	Let $D$ be a definable subset of $(0,\infty)$ with $0 \in \mycl_F(D)$ and $\gamma:D \to F^n$ be a definable map.
	The following assertions hold:
	\begin{enumerate}
		\item[(1)] The equality $\myLim(\gamma)	= \bigcap_{c>0} \mycl_{F^n}(\gamma(D_{<c}))$ holds, where $D_{<c}$ denotes the set $\{x \in D\;|\; x<c\}$.
		\item[(2)] Let $K$ be a definably compact subset of $F^n$. If $\gamma(D) \subseteq K$, the limit set $\myLim(\gamma)$ is not empty and $\myLim(\gamma) \subseteq K$.
		\item[(3)] Let $D'$ be a subset of $D$ with $0 \in \mycl_F(D')$. Let $\gamma'$ be the restriction of $\gamma$ to $D'$.
		Then, we have $\myLim(\gamma') \subseteq \myLim(\gamma)$. 
		\item[(4)]
		Assume that $\mathcal F$ enjoys the definable choice property and admits a good $(0+)$-pseudo-filter $\mathfrak D$.
		Assume further that $\myLim(\gamma) \neq \emptyset$ and let $y \in \myLim(\gamma)$.
		There exist $D' \in \mathfrak D$ and a definable map $\delta:D' \to F^n$ such that $\myLim(\delta)=\{y\}$ and $\delta(D') \subseteq \gamma(D)$.
	\end{enumerate}
\end{lemma}
\begin{proof}
	(1) We first prove the inclusion $\myLim(\gamma)	\subseteq \bigcap_{c>0} \mycl(\gamma(D_{<c}))$.
	The inclusion is obvious when $\myLim(\gamma)=\emptyset$.
	We consider the other case.
	Take an arbitrary point $y \in \myLim(\gamma)$.
	Fix an arbitrary $c>0$.
	For any open box $B$ containing the point $y$, we have $B \cap \gamma(D_{<c})$ is not empty by the definition of $\myLim(\gamma)$.
	It means that $y \in \mycl(\gamma(D_{<c}))$.
	
	We next show the opposite inclusion $\bigcap_{c>0} \mycl(\gamma(D_{<c})) \subseteq \myLim(\gamma)$.
	Take an arbitrary point $y \in \bigcap_{c>0} \mycl(\gamma(D_{<c}))$.
	We fix an arbitrary open box $B$ containing the point $y$ and an arbitrary open interval $I$ containing the point $0$. 
	Since $(0,c) \subseteq I$ for sufficiently small $c>0$, we have $\gamma(D_{<c}) \subseteq \gamma(D \cap I)$ for sufficiently small $c>0$.
	It implies that $y \in \mycl(\gamma(D_{<c})) \subseteq \mycl(\gamma(D \cap I))$ and we have $B \cap \gamma(D \cap I) \neq \emptyset$.
	It means that $\gamma^{-1}(B) \cap I \neq \emptyset$ and we have proven that $y \in \myLim(\gamma)$.
	
	(2) The assertion immediately follows from assertion (1) and Proposition \ref{prop:definably_compact}.
	
	(3) It is obvious by the definition of $\myLim(\gamma)$.
	
	(4) Take an element $D' \in \mathfrak D$.
	By the definition of $\myLim(\gamma)$, for each $c \in F$, the intersection $\gamma^{-1}(\mathcal B(y,c)) \cap D_{<c} $ is not empty, where $\mathcal B(y,c)=\{x \in F^n\;|\; |x-y|<c\}$.
	Since $\mathcal F$ enjoys the definable choice property, there exists a definable map $f:D' \to D$ such that $f(c) \in \gamma^{-1}(\mathcal B(y,c)) \cap D_{<c} $ for each $c \in D'$.
	Set $\delta = \gamma \circ f$.
	It is obvious that $\delta(D') \subseteq \gamma(D)$.
	It is also a routine to show that $y \in \myLim(\delta)$.
	We omit the details.
	
	We next show that $\myLim(\delta)$ is a singleton.
	Let $z \in F^n$ be an arbitrary point with $z \neq y$.
	We can take $\varepsilon>0$ such that $|z-y| > \varepsilon$.
	When $x \in D'$ and $x<\varepsilon/2$, the definition of $f$ implies that $|\gamma\circ f(x)-y|<\varepsilon/2$.
	It means that $|\delta(x)-y|<\varepsilon/2$ and we get $|\delta(x)-z| > \varepsilon/2$.
	It implies that $z \notin \mycl_F(\delta(D'_{<\varepsilon/2}))$, where $D'_{<\varepsilon/2} = \{t \in D'\;|\; t < \varepsilon/2\}$.
	It means that $z \notin \myLim(\delta)$ by assertion (1).    
\end{proof}

\begin{lemma}\label{lem:basic_o+2}
	Consider a definably complete expansion $\mathcal F=(F,+,<,0,\ldots)$ of an ordered group admitting a good $(0+)$-pseudo-filter $\mathfrak D$.
	Let $\gamma:D \to F^n$ be a pseudo-curve.
	The following assertions hold:
	\begin{enumerate}
		\item[(1)] We have $\mycl_{F^n}(\gamma(D))=\gamma(D) \cup \myLim(\gamma)$.
		\item[(2)] Assume that $\myLim(\gamma)$ is a singleton.
		Then $\gamma(D) \cup \myLim(\gamma)$ is definably compact.
		\item[(3)] Assume that $\myLim(\gamma)$ is a singleton.
		For any subset $D'$ of $D$ with $0 \in \mycl_F(D')$, we have $\myLim(\gamma')=\myLim(\gamma)$,  where $\gamma'$ is the restriction of $\gamma$ to $D'$.
	\end{enumerate}
\end{lemma}
\begin{proof}
	(1) The inclusion $\gamma(D) \cup \myLim(\gamma) \subseteq \mycl(\gamma(D))$ is obvious.
	We show the opposite inclusion.
	Take an arbitrary point $x \in \mycl(\gamma(D)) \setminus \myLim(\gamma)$.
	By the definition of $\myLim(\gamma)$, there exist an open box $B$ containing the point $x$ and an open interval containing the point $0$ such that $\gamma^{-1}(B) \cap I$ is an empty set.
	It implies that $\gamma^{-1}(B) \subseteq D \setminus I$.
	We get $\gamma(D) \cap B \subseteq \gamma(D \setminus I)$.
	Note that $D \setminus I$ is definably compact by property (a) in Definition \ref{def:filter}.
	The image $\gamma(D \setminus I)$ is also definably compact by Lemma \ref{lem:image}.
	In particular, it is closed.
	We finally get $x \in \mycl(\gamma(D) \cap B)  \subseteq \gamma(D \setminus I) \subseteq \gamma(D)$.
	We have proven the inclusion $\mycl(\gamma(D)) \subseteq \gamma(D) \cup \myLim(\gamma)$.
	
	(2) Let $y$ be the unique point belonging to $\myLim(\gamma)$.
	Recall that $\gamma$ is continuous by the definition of pseudo-curves.
	Consider the map $f:D \cup \{0\} \to F^n$ given by $f(x)=\gamma(x)$ if $x \in D$ and $f(0)=y$.
	It is a continuous map by the definition of $\myLim(\gamma)$.
	The set $D \cup \{0\}$ is definably compact by property (a) of Definition \ref{def:filter}.
	The image $\gamma(D) \cup \myLim(\gamma)=f(D \cup \{0\})$ is also definably compact by Lemma \ref{lem:image}.
	
	(3) The set $\gamma(D) \cup \myLim(\gamma)$ is definably compact by assertion (2).
	Since $\gamma'(D')$ is contained in the definably compact set $\gamma(D) \cup \myLim(\gamma)$, the set $\myLim(\gamma')$ is not empty by Lemma \ref{lem:basic_o+1}(2).
	We have $\myLim(\gamma')=\myLim(\gamma)$ because $\myLim(\gamma)$ is a singleton and $\myLim(\gamma') \subseteq \myLim(\gamma)$ by Lemma \ref{lem:basic_o+1}(3).
\end{proof}

The following proposition is a counterpart of the curve selection lemma for locally o-minimal structures.
\begin{proposition}\label{prop:curve_selection}
	Consider a definably complete expansion of an ordered group $\mathcal F=(F,<,+,0,\ldots)$ enjoying the definable choice property and admitting a good $(0+)$-pseudo-filter $\mathfrak D$.
	Let $X$ be a definable subset of $F^n$ which is not closed.
	Take a point $p \in \partial_{F^n} X$.
	There exist a pseudo-curve $\gamma:D \to X$ such that $\myLim(\gamma)=\{p\}$.
\end{proposition}
\begin{proof}
	Consider the definable set $Z=\{(x,r) \in X \times F\;|\; r>0, |p-x|<r\}$.
	Let $\pi$ be the restriction of the projection on to the second coordinate to $Z$.
	It is obvious that $\pi(Z)=(0,\infty)$.
	Apply the definable choice property to $\pi$, there exists a definable map $f:(0,\infty) \to X$ such that $|f(r)-p|<r$ for any $r>0$.
	Take $D \in \mathfrak D$.
	It is obvious that $\myLim(f|_D)=\{p\}$.
	By property (c) of Definition \ref{def:filter}, we can take $D' \in \mathfrak D$ such that $D' \subseteq D$, $\gamma:=f|_{D'}$ is a pseudo-curve and $\myLim(\gamma)=\myLim(f|_D)=\{p\}$.
\end{proof}

Preparation has been done.
We begin to give equivalent conditions for a definable map to be continuous/definably proper/definably identifying.
\begin{lemma}\label{lem:cont}
	Let $\mathcal F=(F,<,+,0,\ldots)$ and $\mathfrak D$ be as in Proposition \ref{prop:curve_selection}.
	Let $f:X \to Y$ be a definable map and  $p \in X$.
	The following are equivalent:
	\begin{enumerate}
		\item[(1)] $f$ is continuous at $p$.
		\item[(2)] For any definable subset $D$ of $(0,\infty)$ with $0 \in \mycl_F(D)$ and a definable map $\gamma:D \to X$ with $p \in \myLim(\gamma)$, we have $f(p) \in \myLim(f \circ \gamma)$.
		\item[(3)] For every pseudo-curve $\gamma:D \to X$ with $\myLim(\gamma)=\{p\}$, we have $f(p) \in \myLim(f \circ \gamma)$.
	\end{enumerate}
\end{lemma}
\begin{proof}
	We first show the implication $(1) \Rightarrow (2)$.
	Take an arbitrary open box $B$ containing $f(p)$ and an arbitrary open interval $I$ containing the point $0$.
	Since $f$ is continuous at $p$, there exists an open box $B'$ containing the point $p$ such that $B' \subseteq f^{-1}(B)$.
	Since $p \in \myLim(\gamma)$, we have $\gamma^{-1}(B') \cap I  \neq \emptyset$.
	Since $\gamma^{-1}(B') \subseteq (f \circ \gamma)^{-1}(B)$, we have $(f \circ \gamma)^{-1}(B) \cap I \neq \emptyset$.
	It implies that $f(p) \in \myLim(f \circ \gamma)$.
	
	The implication $(2) \Rightarrow (3)$ is obvious.
	
	We finally show the implication $(3) \Rightarrow (1)$.
	Assume that $f$ is not continuous at $p$.
	There exists an open box $B$ such that $f(p) \in B$ and $f^{-1}(B)$ is not a  neighborhood of $p$.
	Set $Z=X \setminus f^{-1}(B)$.
	We have $p \in \partial Z$.
	There exists a pseudo-curve $\gamma:D \to Z$ such that $\myLim(\gamma)=\{p\}$ by Proposition \ref{prop:curve_selection}.
	Since $f(\gamma(D)) \cap B$ is an empty set, we get $\myLim(f \circ \gamma) \subseteq \mycl(f \circ \gamma(D)) \subseteq \mycl(Y) \setminus B$.
	In particular, we get $f(p) \notin \myLim(f \circ \gamma)$.
\end{proof}

The notion of local boundedness is used in Section \ref{sec:def_quo}.
\begin{definition}
	Consider an expansion of a dense linear order without endpoints $\mathcal F=(F,<,\ldots)$.
	Let $X \subseteq F^m$ and $Y \subseteq F^n$ be definable sets and $f:X \to Y$ be a definable map which is not necessarily continuous.
	We say that $f$ is \textit{locally bounded} if, for any point $x \in X$, there exists a definable open neighborhood $U$ of $x$ in $X$ such that $f(U)$ is bounded.
\end{definition}

\begin{lemma}\label{lem:local_bound}
	Consider a definably complete expansion of an ordered field $\mathcal F=(F,<,+,\cdot,0,1,\ldots)$ enjoying the definable choice property and admitting a good $(0+)$-pseudo-filter $\mathfrak D$.
	Let $K$ be a definably compact set and $f:K \to F$ be a locally bounded definable function.
	Then, $f$ is bounded.
\end{lemma}
\begin{proof}
	We demonstrate the contraposition.
	Assume that $f$ is not bounded.
	We may assume that $f(K)$ is unbounded above by considering $-f$ if necessary.
	By the definable choice property, we can construct a definable map $\rho:(0,\infty) \to K$ such that $f(\rho(t))>1/t$.
	Since $K$ is definably compact, we have $\emptyset \neq \myLim(\rho) \subseteq K$ by Lemma \ref{lem:basic_o+1}(2).
	Take $x \in \myLim(\rho)$.
	The function $f$ is not locally bounded at $x$. 
\end{proof}

The following definitions are found in \cite[Chapter 6, Definition 4.4]{D}.
We study basic properties of definably proper and definably identifying maps in the rest of this section.
\begin{definition}
	We consider a definably complete expansion of a dense linear order without endpoints.
	Let $X$ and  $Y$ be definable sets and $f:X \to Y$ be a definable continuous map.
	The map $f$ is \textit{definably proper} if for any definably compact subset $K$ of $Y$,	the inverse image $f^{-1}(K)$ is definably compact.
	It is \textit{definably identifying} if it is surjective and, for any definable subset $K$ in $Y$, 
	$K$ is closed in $Y$ whenever $f^{-1}(K)$ is closed in $X$.
\end{definition}

\begin{lemma}\label{lem:via_homeo}
	Consider a definably complete structure.
	The following assertions hold true:
	\begin{enumerate}
		\item[(i)] A definable homeomorphism is definably proper.
		\item[(ii)] The composition of definably proper maps is definably proper. 
		\item[(iii)] Let $f:X \to Y$ be a definable proper definable map between definable sets and $S$ be a definable subset of $Y$.
		The restriction $f|_{f^{-1}(S)}$ of $f$ to $f^{-1}(S)$ is also definably proper.
		\item[(iv)] Let $f:X \to Y$ and $g:Y \to Z$ be definable continuous maps.
		Assume that $f$ is surjective and  the composition $g \circ f$ is definably proper.
		Then $g$ is definably proper.
	\end{enumerate}
\end{lemma}
\begin{proof}
	The assertion (i) immediately follows from Lemma \ref{lem:image}.
	The assertions (ii) through (iv) are obvious from the definition of definably proper maps.
	We use Lemma \ref{lem:image} in the proof of (iv).
\end{proof}

\begin{lemma}\label{lem:proper_closed}
	Let $\mathcal F=(F,<,+,\cdot,0,1,\ldots)$ and $\mathfrak D$ be as in Lemma \ref{lem:local_bound}.
A definably proper map $f:X \to Y$ is a definably closed map; that is, $f(C)$ is closed for any definable closed subset $C$ of $X$.
\end{lemma}
\begin{proof}
Let $C$ be a closed definable subset of $X$ and $y$ be a point in the closure of $f(C)$ in $Y$.
We want to show that $y \in f(C)$.

By Proposition \ref{prop:curve_selection}, there exist a pseudo-curve $\gamma:D \to f(C)$ with $\myLim(\gamma)=\{y\}$.
The union $\gamma(D) \cup \{y\}$ is definably compact by Lemma \ref{lem:basic_o+2}(2).
Since $f$ is definably proper, its inverse image $K=f^{-1}(\gamma(D)\cup \{y\}))$ through $f$ is also definably compact.

Consider the definable set $Z=\{(t,x) \in D \times C\;|\; f(x)=\gamma(t)\}$.
By the definition of $\gamma$, the fiber $Z_t:=\{x \in C\;|\;(t,x) \in Z\}$ is not empty for any $t \in D$.
Thanks to the definable choice property, we can choose a definable map $\alpha:D \to C$ such that $f \circ \alpha(t)=\gamma(t)$ for all $t \in D$.
Since the image $\alpha(D)$ is contained in $K$, we have $\myLim(\alpha) \neq \emptyset$ by Lemma \ref{lem:basic_o+1}(2).
We have $\myLim(\alpha) \subseteq C$ because $C$ is closed and $\alpha(D)$  is contained in $C$.
Take $x' \in \myLim(\alpha)$.
We have $f(x') \in \myLim(f \circ \alpha)=\myLim(\gamma)=\{y\}$ by Lemma \ref{lem:cont}(2); that is, $y=f(x')$.
It means that $f(C)$ is closed.
\end{proof}

\begin{lemma}\label{lem:proper_eq}
	Let $\mathcal F=(F,<,+,\cdot,0,1,\ldots)$ and $\mathfrak D$ be as in Lemma \ref{lem:local_bound}.
	Let $X $ and $Y$ be definable sets and $f:X \to Y$ be a definable continuous map.
	The following are equivalent:
	\begin{enumerate}
		\item[(i)]  The map $f$ is definably proper;
		\item[(ii)] Let $\gamma:D \to X$ be a pseudo-curve such that $f \circ \gamma$ is a pseudo-curve and $\myLim(f \circ \gamma)$ is a singleton and contained in $Y$.
		Then $\myLim(\gamma)$ is not an empty set and contained in $X$.
		\item[(iii)] Let $\gamma:D \to X$ be a pseudo-curve such that $f \circ \gamma$ is a pseudo-curve and $\myLim(f \circ \gamma)$ is a singleton and contained in $Y$.
		Then $\myLim(\gamma) \cap X$ is not an empty set.
	\end{enumerate}
\end{lemma}
\begin{proof}
	We first show the implication (i) $\Rightarrow$ (ii).
	Assume that $f$ is definably proper.
	Let $D \in \mathfrak D$ and $\gamma:D \to X$ be a pseudo-curve such that $f \circ \gamma$ is a pseudo-curve and $\myLim(f \circ \gamma)$ is a singleton and contained in $Y$.
	Let $y$ be the unique point belonging to $\myLim(f \circ \gamma)$.
	The definable set $K=f \circ \gamma(D) \cup \{y\}$ is definably compact by Lemma \ref{lem:basic_o+2}(2).
	The inverse image $f^{-1}(K)$ is definably compact because $f$ is definably proper.
	We have $\emptyset \neq \myLim(\gamma) \subseteq f^{-1}(K) \subseteq X$ by Lemma \ref{lem:basic_o+1}(2).
	It implies that $\myLim(\gamma)$ is not an empty set and it is contained in $X$.
	
	The implication (ii) $\Rightarrow$ (iii) is obvious.
	
	The final task is to prove the contraposition of the implication (iii) $\Rightarrow$ (i). 
	We assume that $f$ is not definably proper.
	Let $F^m$ be the ambient space of $X$.
	There is a definably compact subset $K$ of $Y$, but $f^{-1}(K)$ is not closed in $F^m$ or not bounded.
	Let $\pi_i:F^m \to F$ be the coordinate projection onto the $i$-th coordinate for $1 \leq i \leq m$.
	If $f^{-1}(K)$ is not bounded, the image $\pi_i(f^{-1}(K))$ is unbounded for some $1 \leq i \leq m$.
	We may assume that $\sup\pi_1(f^{-1}(K))=\infty$ without loss of generality.
	Set $C=(0,\infty) \cap \pi_1(f^{-1}(K))$.
	Let $\rho:(0,\infty) \to (0,\infty)$ be the map defined by $\rho(t)=1/t$.
	The image $V=\rho(C)$ satisfies the conditions that $0 \in \mycl(V)$ and $V \subseteq (0,\infty)$.
	Take an arbitrary element $D \in \mathfrak D$.
	There exists a definable map $\gamma:D \to f^{-1}(K)$ such that $\rho(\pi_1(\gamma(t))) \in \mathcal B(0,t)$ for each $t \in D$ by the definable choice property because $0 \in \mycl(V)$.
	We may assume that $\gamma$ is continuous by property (c) in Definition \ref{def:filter}.
	Since $K$ is definably compact, $\myLim(f \circ \gamma)$ is not empty and contained in $K$ by Lemma \ref{lem:basic_o+1}(2).
	By property (c) in Definition \ref{def:filter}, we may assume that $f \circ \gamma$ is a pseudo-curve and $\myLim(f \circ \gamma)$ is a singleton by replacing a subset of $D$ belonging to $\mathfrak D$ with $D$.
	On the other hand, we have $$\lim_{D \ni t \to 0}\pi_1(\gamma(t))=\infty.$$
	It means that $\myLim(\gamma)$ is an empty set.
	
	We next consider the case in which $f^{-1}(K)$ is not closed in $F^m$. Take a point $x \in \partial_{F^m} (f^{-1}(K))$.
	Note that $x \notin X$ because $f^{-1}(K)$ is closed in $X$. 
	There exist a pseudo-curve $\gamma:D \to f^{-1}(K)$ such that $\myLim(\gamma)=\{x\}$ by Proposition \ref{prop:curve_selection}.
	The set $\myLim(f \circ \gamma)$ is not empty and contained in $K$ by Lemma \ref{lem:basic_o+1}(2) because $K$ is definably compact and $f \circ \gamma(D) \subseteq K$.  
	Replacing a subset of $D$ belonging to $\mathfrak D$ with $D$, we may assume that $f \circ \gamma$ is a pseudo-curve and $\myLim(f \circ \gamma)$ is a singleton by property (c) in Definition \ref{def:filter}.
	By this replacement, we have $\myLim(\gamma) \subseteq \{x\}$ by Lemma \ref{lem:basic_o+1}(3).
	In summary, $\myLim(f \circ \gamma)$ is a singleton and contained in $Y$, but $\myLim(\gamma) \cap X$ is empty.
	We have shown that condition (iii) is not satisfied in this case.
\end{proof}

\begin{lemma}\label{lem:identifying_eq}
	Let $\mathcal F=(F,<,+,\cdot,0,1,\ldots)$ and $\mathfrak D$ be as in Lemma \ref{lem:local_bound}.
	Let $X$ and $Y$ be definable sets and $f:X \to Y$ be a definable continuous map.
	The following are equivalent:
	\begin{enumerate}
		\item[(i)]  The map $f$ is definably identifying;
		\item[(ii)] Let $\beta:E \to Y$ be a pseudo-curve such that $\myLim(\beta)$ is a singleton and contained in $Y$.
		There exist a pseudo-curve $\alpha:D \to X$ and a point $x \in X$ such that $f \circ \alpha(D) \subseteq \beta(E)$, $\myLim(\alpha)=\{x\}$ and $f(x) \in \myLim(\beta)$.
		\item[(iii)] Let $\beta:E \to Y$ be a pseudo-curve such that $\myLim(\beta)$ is a singleton and contained in $Y$.
		There exist $D \in \mathfrak D$, a definable map $\alpha:D \to X$ and a point $x \in \myLim(\alpha) \cap X$ such that $f \circ \alpha(D) \subseteq \beta(E)$ and $f(x) \in \myLim(\beta)$.
	\end{enumerate}
\end{lemma}
\begin{proof}
We first prove the implication (i) $\Rightarrow$ (ii).
Let $\beta:E \to Y$ be a pseudo-curve such that $\myLim(\beta)$ is a singleton and contained in $Y$.
Let $y$ be the unique point belonging to $\myLim(\beta)$.
Set $V=\{t \in E\;|\; \beta(t)=y\}$ and $W=E \setminus V$.
At least one of $V$ and $W$ contains an element of $\mathfrak D$ by property (b) in Definition \ref{def:filter}.
We consider two cases.

We first consider the case in which $V$ contains an element $D$ of $\mathfrak D$.
Take a point $x \in X$ such that $y=f(x)$. 
It is possible because $f$ is surjective.
Define $\alpha:D \to X$ by $\alpha(t)=x$ for all $x \in D$.
It satisfies condition (ii).

We consider the other case in which $W$ contains an element $D$ of $\mathfrak D$.
Note that $\mycl(\beta(D))=\beta(D) \cup \{y\}$ by Lemma \ref{lem:basic_o+2}(1),(3).
In particular, $\beta(D)$  is not closed in $Y$.
Since $f$ is definably identifying, the inverse image $f^{-1}(\beta(D))$ is not closed in $X$.
On the other hand, $f^{-1}(\beta(D) \cup \{y\})$ is closed in $X$ because $f$ is continuous.
Let $x$ be an element of the frontier of $f^{-1}(\beta(D))$ in $X$.
We have $x \in f^{-1}(\beta(D) \cup \{y\}) \setminus f^{-1}(\beta(D))$ and we get $y=f(x)$.
For any $t \in D$, we have $f^{-1}(\beta(D)) \cap \mathcal B(x,t) \neq \emptyset$.
We can choose a definable map $\alpha:D \to f^{-1}(\beta(D))$ such that $\alpha(t) \in \mathcal B(x,t)$ in the same manner as the proof of Proposition \ref{prop:curve_selection}.
We may assume that $\alpha$ is a pseudo-curve by property (c) in Definition \ref{def:filter} by taking a subset of $D$ belonging to $\mathfrak D$ if necessary.
It is obvious that $\myLim(\alpha)=\{x\}$ and $f \circ \alpha(D) \subseteq \beta(D) \subseteq \beta(E)$.

The implication (ii) $\Rightarrow$ (iii) is obvious.

The remaining task is to prove the implication (iii) $\Rightarrow$ (i).
Assume that $f$ is not definably identifying.
When $f$ is not surjective, nonexistence of $\alpha$ satisfying condition (iii) is obvious when we consider a constant map $\beta:E \to Y$ with $\beta(E) \cap f(X) = \emptyset$.
Therefore, we may assume that $f$ is surjective.
There exists a definable subset $K$ of $Y$ such that $K$ is not closed in $Y$ and $f^{-1}(K)$ is closed in $X$.
Take a point $y \in \partial_Y K$.
By Proposition \ref{prop:curve_selection}, there exist a pseudo-curve $\beta:E \to K$ such that $\myLim(\beta)=\{y\}$.
We assume that condition (iii) is satisfied.
By the assumption, there exist a pseudo-curve $\alpha:D \to X$ and $x \in  \myLim(\alpha) \cap X$ such that $f \circ \alpha(D) \subseteq \beta(E)$ and $y=f(x)$.
In particular, we have $\alpha(D) \subseteq f^{-1}(K)$.
Since $x \in \mycl_X(\alpha(D))$ and $f^{-1}(K)$ is closed, we have $x \in f^{-1}(K)$.
It implies that $y=f(x) \in K$.
It is a contradiction.
\end{proof}

\begin{lemma}\label{lem:identifying_basic}
	Let $\mathcal F=(F,<,+,\cdot,0,1,\ldots)$ and $\mathfrak D$ be as in Lemma \ref{lem:local_bound}.
	The following assertions hold true:
	\begin{enumerate}
		\item[(1)] The composition of two definably identifying maps is definably identifying.
		\item[(2)] Let $f:X \to Y$ and $g:Y \to Z$ be definable continuous maps.
		If the composition $g \circ f$ is definably identifying, the map $g$ is so.
	\end{enumerate}
\end{lemma}
\begin{proof}
	We first prove assertion (1).
	Let $f:X \to Y$ and $g:Y \to Z$ be definably identifying maps.
	The composition $g \circ f$ is obviously surjective.
	Let $C$ be a definable subset of $Z$ such that $(g \circ f)^{-1}(C)$ is closed.
	Since $f$ is definably identifying, $g^{-1}(C)$ is closed.
	Then, $C$ is closed because $g$ is definably identifying.
	
	We next prove assertion (2).
	It is obvious that $g$ is surjective.
	We use Lemma \ref{lem:identifying_eq}.
	Let $\beta:E \to Z$ be a pseudo-curve such that $\myLim(\beta)$ is singleton and contained in $Z$.
	Since $g \circ f$ is definably identifying, there exist $D \in \mathfrak D$, a definable map $\gamma:D \to X$ and a point $x \in \myLim(\alpha) \cap X$  such that $g \circ f \circ \gamma (D) \subseteq  \beta(E)$ and $g \circ f(x) \in \myLim(\beta)$ by Lemma \ref{lem:identifying_eq}.
	Consider the pseudo-curve $\alpha = f \circ \gamma:D \to Y$.
	We have $y:=f(x) \in \myLim(\alpha) \cap Y$ by Lemma \ref{lem:cont}.
	We get $g(y)=g\circ f(x) \in \myLim(\beta)$ by Lemma \ref{lem:cont} and we obtain $g \circ \alpha(D)=  g \circ f \circ \gamma (D) \subseteq  \beta(E)$.
	It implies $g$ is definably identifying by Lemma \ref{lem:identifying_eq}.
\end{proof}

\begin{corollary}\label{cor:proper_identifying}
	Let $\mathcal F=(F,<,+,\cdot,0,1,\ldots)$ and $\mathfrak D$ be as in Lemma \ref{lem:local_bound}.
	A definably proper surjective map is definably identifying.
\end{corollary}
\begin{proof}
	Let $f:X \to Y$ be a definably proper surjective map.
	Let $\beta:D \to Y$ be a pseudo-curve such that $\myLim(\beta)$ is a singleton and contained in $Y$.
	By the definable choice property and surjectivity of $f$, there exists a definable map $\alpha:D \to X$ such that $\beta=f \circ \alpha$.
	We may assume that $\alpha$ is continuous as usual using property (c) in Definition \ref{def:filter}.
	Since $f$ is definably proper, $\myLim(\alpha)$ is not empty and contained in $X$ by Lemma \ref{lem:proper_eq}.
	Take an arbitrary point $x \in \myLim(\alpha)$.
	We have $f \circ \alpha(D)=\beta(D)$ and $f(x) \in \myLim(\beta)$ by Lemma \ref{lem:cont}.
	It implies that $f$ is definably identifying by Lemma \ref{lem:identifying_eq}.
\end{proof}

In the last of this section, we give another equivalent condition for a definable map to be definably proper.
We do not use the following proposition in this paper, but it is worth to be mentioned. 
\begin{proposition}\label{prop:properness}
	Let $\mathcal F=(F,<,+,\cdot,0,1,\ldots)$ and $\mathfrak D$ be as in Lemma \ref{lem:local_bound}.
	A definable continuous map $f:X \to Y$ is definably proper if and only if it is a definably closed map and $f^{-1}(y)$ is definably compact for any $y \in Y$.
\end{proposition}
\begin{proof}
	The `only if' part follows from the definition of definable properness and Lemma \ref{lem:proper_closed}.
	We concentrate on the `if' part.
	
	Let $F^m$ and $F^n$ be the ambient spaces of $X$ and $Y$, respectively.
	Fix a definably compact subset $C$ of $Y$.
	We have only to demonstrate that the inverse image $f^{-1}(C)$ is closed and bounded in $F^m$.
	
	We first demonstrate that $f^{-1}(C)$ is closed in $F^m$.
	Assume for contradiction that $f^{-1}(C)$ is not closed in $F^m$.
	Take a point $p \in \partial_{F^m}(f^{-1}(C))$.
	Note that $f^{-1}(C)$ is closed in $X$ because $f$ is continuous.
	It implies that $p \not\in X$. 
	There exists a pseudo-curve $\gamma:D\to f^{-1}(C)$ such that $\myLim(\gamma) = \{p\}$ by Proposition \ref{prop:curve_selection}.
	Since $C$ is definably compact, we may assume that $\myLim(f \circ \gamma)$ is a singleton and contained in $C$ by Lemma \ref{lem:basic_o+1}(2) and property (c) in Definition \ref{def:filter} by replacing $D$ with an appropriate subset of $D$ belonging to $\mathfrak D$.
	Let $q$ be a unique point belonging to $\myLim(f \circ \gamma)$.
	Set $V:=\{t \in D\;|\; f \circ \gamma(t)=q\}$ and $W=D \setminus V$.
	At least one of $V$ and $W$ contains an element, say $D'$, of $\mathfrak D$ by property (b) in Definition \ref{def:filter}.
	Let $\gamma'$ be the restriction of $\gamma$ to $D'$.
	We have $\myLim(\gamma')=\{p\}$ and $\myLim(f \circ \gamma')=\{q\}$ by Lemma \ref{lem:basic_o+2}(3).
	We consider two separate cases.
	
	The first case is the case in which $D'$ is contained in $V$.
	Since $f^{-1}(q)$ is definably compact, we have $p \in f^{-1}(q)$.
	In particular, we have $p \in X$.
	It is a contradiction.
	
	We next consider the case in which $D'$ is contained in $W$.
	We have $q \in \myLim(f \circ \gamma')  \setminus f \circ \gamma'(D') \subseteq \mycl(f \circ \gamma'(D')) \setminus f \circ \gamma'(D')$ by Lemma \ref{lem:basic_o+2}(3) and the definition of $D'$.
	We also have $q \in C \subseteq Y$ because $C$ is closed and $f \circ \gamma(D) \subseteq C$.
	They imply that the image $f \circ \gamma'(D')$ is not closed in $Y$.
	On the other hand, $\gamma'(D')$ is closed in $X$ because $p \notin X$ and $\myLim(\gamma')=\{p\}$ by Lemma \ref{lem:basic_o+2}(3).
	The image $f \circ \gamma'(D')$ is closed in $Y$ because $f$ is a definably closed map.
	It is a contradiction.
	
	We next demonstrate that $f^{-1}(C)$ is bounded.
	Consider the definable function $\rho:C \to F$ given by $\rho (y)=\sup \{|x|\;|\; x \in f^{-1}(y)\}$.
	The function $\rho$ is well-defined because $f^{-1}(y)$ is definably compact.
	We prove that $\rho$ is locally bounded.
	Assume for contradiction that $\rho$ is not locally bounded at $\hat{y} \in C$.
	By the definition of $\rho$, for any $t>0$, there exist $y \in C$ and $x \in X$ 
	such that
	$|y-\hat{y}|<t$, $|x|>\frac{1}{t}$, $y=f(x)$ and $y \neq \hat{y}$.
	Put $$Z=\{(t,x, y) \in F \times X \times C\;|\;t>0, |x|>\frac{1}{t}, y=f(x), |y - \hat{y}|<t \text{ and }
	y \neq \hat{y}\}.$$
	Let $p:F \times X \times C \to F$ be the projection.
	We get $p(Z)=\{t>0\}$.
	By the definable choice property,
	there exist definable maps $g:(0, \infty) \to X$ and $h:(0, \infty) \to C$ such that 
	$\Gamma (g \times h)\subseteq Z$, where $\Gamma(g \times h)$ is the graph of the map $g \times h$ defined by $t \mapsto (g(t),h(t))$.
	We obviously have $h=f \circ g$.
	Since $|g(t)|>1/t$ for $t>0$, we get $\myLim(g)=\emptyset$.
	Take $D \in \mathfrak D'$ so that the restriction $g'$ of $g$ to $D$ is a pseudo-curve using property(c) in Definition \ref{def:filter}.
	Set $h'=f \circ g'$.
	We have $\myLim(g')=\emptyset$ by Lemma \ref{lem:basic_o+1}(3).
	The image $g'(D)$ is closed in $F^m$ by Lemma \ref{lem:basic_o+2}(1).
	It means that $g'(D)$ is closed in $X$.
	
	On the other hand, we have $\myLim(f \circ g')=\myLim(h')=\{\hat{y}\}$ and $\hat{y} \notin h'(D) = f \circ g'(D)$ by the definitions of $Z$ and $h'$.
	It implies that $f \circ g'(D)$ is not closed in $Y$ by Lemma \ref{lem:basic_o+2}(1).
	Since $f$ is a definably closed map, $g'(D)$ is not closed in $X$.
	It is a contradiction.
	We have proven that $\rho$ is locally bounded.
	
	By Lemma \ref{lem:local_bound}, 
	$\rho$ is bounded.
	Therefore, $f^{-1}(C)$ is bounded.
\end{proof}

\section{Definable proper quotients}\label{sec:def_prop_quo}

We first introduce the definitions of definable quotient and definable proper quotient.
\begin{definition}
	Consider an expansion of a dense linear order without endpoints.
	Let $X$ be a definable set and $E$ be a definable equivalence relation on it.
	A \textit{definable quotient of $X$ by $E$} is 
	a definably identifying map $f:X \to Y$ such that $f(x)=f(x')$ if and only if $(x,x') \in E$.
	A \textit{definably proper quotient of $X$ by $E$} is a definable proper map $f:X \to Y$  satisfying that $f(x)=f(x') \Leftrightarrow (x,x') \in E$.
	The target space $Y$ is sometimes denoted by $X/E$.
	Note that a definable proper quotient is a definable quotient by Corollary \ref{cor:proper_identifying}.
\end{definition}

We next recall the definition of definably proper definable equivalence relation.
\begin{definition}
	Consider an expansion of a dense linear order without endpoints.
	Let $E \subseteq X \times X$ be a definable equivalence relation on a definable set $X$.
	Let $p_i:E \to X$ be the restriction of the projection $X \times X \to X$ onto the $i$-th factor to $E$ for $i=1,2$.
	The definable equivalence relation $E$ is \textit{definably proper over $X$} if $p_1$ (equivalently, $p_2$) is a definably proper map.
\end{definition}

We introduce several technical definitions.
\begin{definition}
	Consider an expansion of a dense linear order without endpoints $\mathcal F=(F,<,\ldots)$.
	Let $S_1,\ldots, S_k$ be definable sets in $F^{m_1},\ldots, F^{m_k}$ for $k \geq 1$.
	A \textit{disjoint sum} of $S_1,\ldots, S_k$ is a tuple $(h_1,\ldots, h_k,T)$ consisting of a definable set $T \subseteq F^n$ and definable maps $h_i:S_i \to T$ such that
	\begin{enumerate}
		\item[(i)] $h_i$ is a definable homeomorphism onto $h_i(S_i)$ and $h_i(S_i)$ is open in $T$ for $1 \leq i \leq k$;
		\item[(ii)] $T$ is the disjoint union of the sets $h_1(S_1), \ldots, h_k(S_k)$.
	\end{enumerate}
	
	The existence and the uniqueness up to definable homeomorphism of the disjoint sum are easy to be proven. 
	We omit their proofs here.
\end{definition}

\begin{definition}
	Consider an expansion of a dense linear order without endpoints.
	Let $X$, $Y$ and $A$ be definable sets with $A \subseteq X$.
	Let $f:A \to Y$ be a definable continuous map.
	We define the definable equivalence relation $E(f)$ on $X \amalg Y$ by
	\begin{align*}
		E(f) =&\Delta(X) \cup \Delta(Y) \cup \{(a,f(a))\;|\; a \in A\} \cup \{(f(a),a)\;|\; a \in A\}\\
		&\cup\{(a_1,a_2) \in A \times A\;|\; f(a_1)=f(a_2)\},
	\end{align*}
	where $\Delta(X):=\{(x,x) \in X \times X\;|\; x \in X\}$ and $\Delta(Y):=\{(y,y) \in Y \times Y\;|\;y \in Y\}$ are the diagonals of $X$ and $Y$, respectively.
	Let $X \amalg_f Y$ be the definable quotient of $X \amalg Y$ by $E(f)$ if it exists.
\end{definition}

We demonstrate that definable proper quotient exists for the definable equivalence relation defined above.
\begin{lemma}\label{lem:attach}
	Consider a definably complete expansion of an ordered field $\mathcal F=(F,<,+,\cdot,0,1,\ldots)$ enjoying the definable choice property and admitting a good $(0+)$-pseudo-filter $\mathfrak D$.
	Let $X$, $Y$ and $A$ be definable sets with $A \subseteq X$.
	Let $f:A \to Y$ be a definable continuous map.
	Assume that $A$ is definably compact.
	Then $X \amalg_f Y$ exists as a definable proper quotient.
\end{lemma}
\begin{proof}
	If $A$ is an empty set, the identity map $X \amalg Y \to X \amalg Y$ is a definable proper quotient of $X \amalg Y$ by $E(f)$.
	So we assume that $A$ is not an empty set.
	We identify $X$ and $Y$ with their images in $X \amalg Y \subseteq F^n$.
	Let $d_A:F^n \to F$ be the distance function defined by $d_A(x):= \inf\{|x-a|\;|\; a \in A\}$.
	We also get a definable continuous extension $\widetilde{f}:X \to F^n$ of the definable continuous map $f:A \to Y$.
	It exists by Theorem \ref{thm:tietze} because $f(A)$ is bounded by Lemma \ref{lem:image}.
	Define the definable continuous map $p:X \amalg Y \to F^{2n+1}$ by
	\begin{align*}
		p(x)= \left\{\begin{array}{ll}
			(\widetilde{f}(x),d_A(x)\cdot x, d_A(x)) & \text{ if }x \in X,\\
			(x,0,0) & \text{ if } x \in Y.
		\end{array}
		\right.
	\end{align*} 
	Set $Z=p(X \amalg Y) $.
	We want to show that the induced map $p:X \amalg Y \to Z$ is a definable proper quotient.
	
	It is easy to check that $(z_1,z_2) \in E(f)$ if and only if $p(z_1)=p(z_2)$ for any $z_1,z_2 \in X \amalg Y$.
	We omit the details.
	The remaining task is to demonstrate that $p:X \amalg Y \to Z$ is definably proper.
	Let $\gamma:D \to X \amalg Y$ be a pseudo-curve such that $\myLim(p \circ \gamma)$ is a singleton and contained in $Z$.
	Let $z \in Z$ be the unique point belonging to $\myLim(p \circ \gamma)$.
	We have only to demonstrate that $\myLim(\gamma) \cap (X \amalg Y)$ is not an empty set by Lemma \ref{lem:proper_eq}.
	For latter use, we define the projections $\pi_1:F^n \times F^n \times F \to F^n$, $\pi_2:F^n \times F^n \times F \to F^n$ and $\pi_3:F^n \times F^n \times F \to F$.
	The projection $\pi_i$ is the projection of $F^n \times F^n \times F$ onto the $i$-th factor for each $1 \leq i \leq 3$.
	
	Set $D_X:=\{t \in D\;|\; \gamma(t) \in X\}$ and $D_Y:=\{t \in D\;|\; \gamma(t) \in Y\}$.
	By property (b) in Definition \ref{def:filter}, at least one of $D_X$ and $D_Y$ contains an element $D' \in \mathfrak D$.
	By replacing $D$ with $D'$, we may assume that the image of $\gamma$ is entirely contained in either $X$ or $Y$ by Lemma \ref{lem:basic_o+1}(3).
	We still have $\myLim(p \circ \gamma)=\{z\}$ by Lemma \ref{lem:basic_o+2}(3) even when $D$ is replaced with $D'$.
	The case in which $\gamma(D) \subseteq Y$ is simple.
	We first consider this case.
	By the definition of $p$, the composition $\pi_1 \circ p|_{Y}$ is the identity map on $Y$.
	We have $\pi_1(z) \in \myLim(\gamma)$ by Lemma \ref{lem:cont} because $\pi_1$ is continuous and $\gamma(t)=\pi_1(p(\gamma(t)))$ for $t \in D$ in this case.
	We have $\pi_1(z) \in \pi_1 \circ p(Y)=Y$ because $z \in p(Y)$.
	It means that $\myLim(\gamma) \cap (X \amalg Y)$ is not an empty set.
	
	We next consider the case in which $\gamma(D) \subseteq X$.
	In the same way as above, we may assume that either $\gamma(D) \subseteq A$ or $\gamma(D) \subseteq X \setminus A$.
	In the former case, $\myLim(\gamma)$ is not an empty set and contained in $A$ by Lemma \ref{lem:basic_o+1}(2).
	The remaining case is the case in which $\gamma(D) \subseteq X \setminus A$.
	Set $d=\inf\{d_A(\gamma(t))\;|\;t \in D\}$.
	We consider two separate cases in which $d=0$ and $d>0$.
	We first consider the case in which $d>0$.
	We have $z \in p(X \setminus A)$ by the definition of $d$.
	The map $x \mapsto \pi_2(p(x))/\pi_3(p(x))$ is the identity map when $x \in X \setminus A$.
	 Since the map $p(X \setminus A) \ni y \mapsto \pi_2(y)/\pi_3(y) \in X$ is continuous, we have $\pi_2(z)/\pi_3(z) \in \myLim(\gamma) \cap (X \setminus A)$ in this case by Lemma \ref{lem:cont} and $\myLim(\gamma) \cap (X \amalg Y)$ is not an empty set.
	
	The last case is the case in which $d=0$.
	In this case, we obviously have $\lim_{D \ni t \to 0} d_A(\gamma(t)) = 0$.
	For any $t \in D$, the definable set $\{a \in A\;|\; |\gamma(t)-a|=d_A(\gamma(t))\}$ is not empty because $A$ is definably compact.
	By the definable choice property, we can find a definable map $\rho:D \to A$ such that $d_A(\gamma(t))=|\rho(t)-\gamma(t)|$ for each $t \in D$.
	The set $\myLim(\rho)$ is not an empty set and contained in $A$ by Lemma \ref{lem:basic_o+1}(2).
	It is obvious that $\myLim(\rho)=\myLim(\gamma)$ by the equalities $d_A(\gamma(t))=|\rho(t)-\gamma(t)|$ and $\lim_{D \ni t \to 0} d_A(\gamma(t)) = 0$.
	It implies that $\myLim(\gamma) \cap (X \amalg Y)$ is not an empty set.
\end{proof}

\begin{proposition}\label{prop:attach}
	Let $\mathcal F=(F,<,+,\cdot,0,1,\ldots)$ and $\mathfrak D$ be as in Lemma \ref{lem:attach}.
	Let $X$, $Y$ and $A$ be definable sets with $A \subseteq X$.
	Let $f:A \to Y$ be a definably proper map.
	Assume that $A$ is closed in $X$.
	Then $X \amalg_f Y$ exists as a definable proper quotient.
\end{proposition}
\begin{proof}
	Let $F^m$ and $F^n$ be the ambient spaces of $X$ and $Y$, respectively.
	We first reduce to the case in which $X$ and $Y$ are bounded, and $f$ is the restriction of the coordinate projection onto the last $n$ coordinates to $A$.
	We have to check that a new $f$ is still definably proper every time we replace $X$, $Y$ and $A$ with other sets using Lemma \ref{lem:via_homeo}.
	We omit the checks in the proof.
	
	Let $\psi:F \to (-1,1)$ be the definable homeomorphism given by $\psi(x)=\dfrac{x}{\sqrt{1+x^2}}$.
	Let $\psi_m:F^m \to (-1,1)^m$ be the map given by $\psi_m(x_1,\ldots,x_m)=(\psi(x_1),\ldots, \psi(x_m))$ and we define $\psi_n$ similarly.
	Replacing $X$ and $Y$ with $\psi_m(X)$ and $\psi_n(Y)$, we may further assume that $X$ and $Y$ are bounded.
	There exists a definable continuous extension $\widetilde{f}:X \to F^n$ of $f$ by Theorem \ref{thm:tietze} because $f$ is bounded.
	Consider the graph $\Gamma(\widetilde{f}) \subseteq F^{m+n}$ of $\widetilde{f}$.
	Let $\pi_1:F^{m+n} \to F^m$ and $\pi_2:F^{m+n} \to F^n$ be the coordinate projections onto the first $m$ coordinates and onto the last $n$ coordinates, respectively.
	Set $X'=\Gamma(\widetilde{f})$ and $A'=\pi_1^{-1}(A) \cap \Gamma(\widetilde{f})$.
	We may assume that $f$ is the restriction of $\pi_2$ to $A$ considering $X'$ and $A'$ instead of $X$ and $A$, respectively.
	
	Let $\mycl(f):\mycl(A) \to \mycl(Y)$ be the restriction of the projection $\pi_2$ to $\mycl(A)$.
	It is a definable continuous extension of $f$ to $\mycl(A)$ because $f$ is the restriction of $\pi_2$ to $A$ by the assumption.
	We demonstrate that 
	\begin{equation}
		\mycl(f)^{-1}(Y)=A. \label{eq:1}
	\end{equation}
	Let $x$ be a point in $\mycl(A)$ with $\mycl(f)(x) \in Y$.
	We have only to show that $x \in A$.
	By Proposition \ref{prop:curve_selection}, there exists a pseudo-curve $\gamma:D \to A$ such that $\myLim(\gamma)=\{x\}$.
	It is obvious that $x \in A$ when $x=\gamma(t)$ for some $t \in D$.
	Therefore, we may assume that $x \notin \gamma(D)$.
	Since $\mycl(f)$ is an extension of $f$, we have $\mycl(f)(\gamma(D))=f(\gamma(D)) \subseteq Y$.
	Since $\mycl(f)$ is continuous and $\gamma(D) \cup \{x\}$ is definably compact by Lemma \ref{lem:basic_o+2}(2), the image $\mycl(f)(\gamma(D) \cup \{x\})=f\circ \gamma(D) \cup \{\mycl(f)(x)\}$ is definably compact by Lemma \ref{lem:image}.
	In particular, $f\circ \gamma(D) \cup \{\mycl(f)(x)\}$ is closed and it contains $\myLim(f \circ \gamma)$.
	In addition, $\myLim(f \circ \gamma)$ is a nonempty set by Lemma \ref{lem:basic_o+1}(2). 
	We may assume that $\myLim(f \circ \gamma)$ is a singleton by property (c) in Definition \ref{def:filter} by replacing $D$ with a subset of $D$ belonging to $\mathfrak D$.
	After the replacement, we still have $\myLim(\gamma)=\{x\}$ by Lemma \ref{lem:basic_o+2}(3).
	We have $\myLim(f \circ \gamma)=\myLim(\mycl(f) \circ \gamma) \ni \mycl(f)(x)$ by Lemma \ref{lem:cont} because $\mycl(f)$ is continuous.
	They imply that $\myLim(f \circ \gamma)=\{ \mycl(f)(x)\}$.
	Since $f$ is definably proper and $A$ is closed in $X$, we have $\myLim(\gamma) \subseteq A$ by Lemma \ref{lem:proper_eq}.
	It implies that $x \in A$.
	We have demonstrated equality (\ref{eq:1}).
	
	There exists a definable proper quotient $\mycl(p):\mycl(X) \amalg \mycl(Y) \to \mycl(X) \amalg_{\mycl(f)} \mycl(Y)$ by Lemma \ref{lem:attach} because $\mycl(A)$ is definably compact.
	We naturally identify $X \amalg Y$ with a subset of $\mycl(X) \amalg \mycl(Y)$.
	Set $Z=\mycl(p)(X \amalg Y)$.
	It is easy to check that, if $x \in X \amalg Y$ and $(x,y) \in E(\mycl(f))$, we have $y \in X \amalg Y$ by equality (\ref{eq:1}).
	It implies that $\mycl(p)^{-1}(Z)=X \amalg Y$ and $E(\mycl(f)) \cap ((X \amalg Y)\times(X \amalg Y))=E(f)$.
	Hence, $p:=\mycl(p)|_{X \amalg Y}:X \amalg Y \to Z$ is a definable proper quotient of $X \amalg Y$ by $E(f)$ by Lemma \ref{lem:via_homeo}(iii). 
\end{proof}

\begin{lemma}\label{lem:complete}
	Let $\mathcal F=(F,<,+,\cdot,0,1,\ldots)$ and $\mathfrak D$ be as in Lemma \ref{lem:attach}.
	Let $X$ be a definable subset of $F^n$ and $E \subseteq X \times X$ be a definable equivalence relation on $X$ which is definably proper over $X$.
	Let $p_i:E \to X$ be the restriction of the projection $X \times X \to X$ onto the $i$-th factor to $E$ for $i=1,2$.
	Let $D \in \mathfrak D$ and $\gamma:D \to E$ be a definable map, and set $\alpha = p_1 \circ \gamma$ and $\beta = p_2 \circ \gamma$.
	If either $\myLim(\alpha)$ or $\myLim(\beta)$ is nonempty and contained in $X$, the intersection of the other with $X$ is nonempty.
	Furthermore, if $\myLim(\alpha)=\{p\} \subseteq X$ there exists $q \in X$  such that $(p,q) \in \myLim(\gamma) \cap E$. 
\end{lemma}
\begin{proof}
	We first prove the `furthermore' part.
	We may assume that $\gamma$ is a pseudo-curve by taking an appropriate definable subset of $D$ by property (c)  in Definition \ref{def:filter}.
	Since $p_1$ is definably proper, $\myLim(\gamma)$ is not empty and contained in $E$ by Lemma \ref{lem:proper_eq}.
	Let $z \in \myLim(\gamma)$.
	We have $p=p_1(z)$ by Lemma \ref{lem:cont}.
	It is obvious that $q=p_2(z) \in \myLim(\beta)$ by the same lemma because $p_2$ is continuous.
	We have $z=(p,q) \in \myLim(\gamma) \cap E$. 
	
	We next prove the main part of the lemma.
	Assume that $\myLim(\alpha)$ is not empty and contained in $X$.
	By property (c) in Definition \ref{def:filter}, we can choose a subset $D'$ of $D$ belonging to $\mathfrak D$ so that $\myLim(\alpha')$ is a singleton, where $\alpha'$ is the restriction of $\alpha$ to $D'$.  
	Apply the 'furthermore' part to the restriction $\gamma'$ of $\gamma$ to $D'$.  
	The set $\myLim(\beta') \cap X$ is nonempty, where $\beta'$ is the restriction of $\beta$ to $D'$.
	The set $\myLim(\beta) \cap X$ is also nonempty by Lemma \ref{lem:basic_o+1}(3).
\end{proof}

We are ready to prove the following theorem:
\begin{theorem}\label{thm:quotient-mae}
	Consider a definably complete expansion of an ordered field $\mathcal F=(F,<,+,\cdot,0,1,\ldots)$ enjoying the definable choice property and admitting a good $(0+)$-pseudo-filter $\mathfrak D$.
	Assume further that $\mathcal F$ has a good extended rank function $\myedim$.
	Let $X$ be a nonempty definable set and $E$ be a definable equivalence relation on $X$ which is definably proper over $X$.
	Then there exists a definable proper quotient $\pi:X \to X/E$.
\end{theorem}
\begin{proof}
	Let $d=\myedim(X)$.
	The first task is to prove that there exist a definable subset $S$ of $X$ and a definable map $\sigma:X \to S$ such that 
	\begin{enumerate}
		\item[(i)]  $\sigma(x)=\sigma(y)$ if and only if $(x,y) \in E$ ;
		\item[(ii)] $\sigma$ is an identity map on $S$;
		\item[(iii)] $\myedim(B)<d$, where $B=\sigma(\partial_XS)$.
	\end{enumerate}
	Thanks to the definable choice property, there exists a definable subset $S$ of $X$ such that $S$ intersects at exactly one point with each equivalence class of $E$. 
	Let $\sigma:X \to S$ be the definable map defined by assigning $x$ to the unique point in $S$ to which it is equivalent.
	Note that the definable map $\sigma$ is uniquely determined and satisfies conditions (i) and (ii).
	If $\myedim(B)<d$, we have succeeded in choosing $S$ and $\sigma$ satisfying conditions (i) through (iii).
	
	We consider the case in which $\myedim(B) =d$.
	By Lemma \ref{lem:contained}, we have $\myedim(\partial_XS) < d$.
	By the definition of $B$, for each $x \in B$, there exists $x' \in \partial_XS$ such that $\sigma(x')=x$.
	By the definable choice property, there exists a definable map $\tau:B \to \partial_XS$ such that $\sigma \circ \tau$ is an identity map on $B$.
	Set $S'=\tau(B) \cup (S \setminus B)$ and let $\sigma':X \to S'$ be the surjective definable map defined by $\sigma'(x)=\tau(\sigma(x))$ if $\sigma(x) \in B$ and $\sigma'(x)=\sigma(x)$ otherwise.
	It is obvious that the pair $(S',\sigma')$ satisfies conditions (i) and (ii).
	We have 
	\begin{align*}
		\mycl_X(S') &\subseteq
		\mycl_X(\partial_XS) \cup \mycl_X(S \setminus B) \subseteq  (\mycl_X(S) \setminus \myint_X(S)) \cup (\mycl_X(S) \setminus \myint_X(B))\\
		&=\mycl_X(S) \setminus \myint_X(B)
	\end{align*}
	because $\tau(B) \subseteq \partial_XS$ and $B \subseteq S$.
	Therefore, we get 
	\begin{align*}
		\mycl_X(S') \setminus S' &\subseteq (\mycl_X(S) \setminus \myint_X(B))  \setminus (S \setminus B) \\
		&=   (\mycl_X(S) \setminus S) \cup (B \setminus \myint_X(B)) \subseteq (\mycl_X(S) \setminus S) \cup B.
	\end{align*}
	We obtain 
	\begin{align*}
		\sigma'(\mycl_X(S') \setminus S') &\subseteq \sigma'(\mycl_X(S) \setminus S) \cup \sigma'(B ) =\tau \circ \sigma(\mycl_X(S) \setminus S) \cup \tau \circ \sigma(B)\\
		&=\tau(B) \cup \tau(B)=\tau(B) \subseteq \mycl_X(S) \setminus S.
	\end{align*}
	We have $\myedim(\sigma'(\mycl_X(S') \setminus S')) \leq \myedim(\mycl_X(S) \setminus S) < d$ by property (d) in Definition \ref{def:extended_rank}.
	The pair $(S',\sigma')$ satisfies conditions (i) through (iii).
	Replacing $(S,\sigma)$ with $(S',\sigma')$, we may assume that $(S,\sigma)$ satisfies conditions (i) through (iii).

	We prove the theorem by induction on $d$.
	The induction is possible thanks to property (c) in Definition \ref{def:extended_rank}.
	When $d$ is the smallest element in $\mathcal E_n$, the definable set $X$ is discrete and closed by property (f) in Definition \ref{def:extended_rank} and Lemma \ref{lem:contained}.
	The map $\sigma$ is continuous because any function on a discrete set is continuous.
	Since $E$ is definably proper over $X$, the map $\sigma$ is definably proper.
	In fact, let $T$ be a definably compact definable subset of $S$.
	We have $\sigma^{-1}(T)=p_2 (p_1^{-1}(T))$, where $p_1$ and $p_2$ are maps defined in Lemma \ref{lem:complete}.
	The inverse image $p_1^{-1}(T)$ is definably compact because $E$ is definably proper.
	The image $p_2 (p_1^{-1}(T))$ is definably compact by Lemma \ref{lem:image}.
	We have proven that $\sigma$ is a definable proper quotient of $X$ by $E$.
	The case in which $d$ is the smallest element is completed.
	
	We consider the other case.
	When $\myedim S=d$, we have $\myedim(S \setminus B)=d$ by property (d) in Definition \ref{def:extended_rank} because $\myedim(B)<d$.
	We can take a definable subset $S_d$ of $S$ such that $\myedim (S \setminus S_d)<d$, $S_d \cap  B=\emptyset$ and $S_d$ is open in $X$ by properties (d) and (e) in Definition \ref{def:extended_rank}.
	When $\myedim S<d$, the empty set $S_d=\emptyset$ satisfies the above conditions. 
	Put $S':=\mycl_X(S) \setminus S_d$.
	We get $\myedim(S')<d$ by Lemma \ref{lem:contained} and property (d) in Definition \ref{def:extended_rank}.
	Since $B \cap S_d$ is an empty set, no point of $S_d$ is equivalent to a point of $S'$.
	Moreover, $S'$ is closed in $X$.
	The definable equivalence relation $E':=E \cap (S' \times S')$ is definably proper over $S'$ by Lemma \ref{lem:via_homeo}(iii) because $E$ is definably proper over $X$.
	Apply the induction hypothesis to the definable equivalence relation $E'$ on $S'$.
	There exists a definable proper quotient $f':S' \to Y'$ of $S'$ by $E'$.
	Note that $f'$ maps $S \setminus S_d$ bijectively onto $Y'$.
	The injectivity follows from the definition of $S$ and the surjectivity follows from the equality $B \cap S_d=\emptyset$.
	Put $A:=\mycl_X(S_d) \cap S'$.
	Note that $A$ is closed in $\mycl_X(S_d)$ and the restriction $f'':=f'|_{A}:A \to Y'$ is definably proper.
	We can apply Proposition \ref{prop:attach} to obtain a definable proper quotient $p:\mycl_X(S_d) \amalg Y' \to \mycl_X(S_d) \amalg_{f''} Y'$.
	Set $Y=\mycl_X(S_d) \amalg_{f''} Y'$.
	
	Note that the composed map $p \circ f':S' \to Y' \to Y$ agrees with the restriction $p|_{\mycl_X(S_d)}:\mycl_X(S_d) \to Y$ on the intersection $A$ of their domains.
	Hence these two maps determine a definable continuous map $g:\mycl_X(S)=\mycl_X(S_d) \cup S' \to Y$.	
	Consider the following commutative diagram:
	
	\[
	\begin{CD}
		\mycl_X(S_d) \amalg S' @>{j}>> \mycl_X(S_d) \amalg Y' \\
		@V{h}VV    @VV{p}V \\
		\mycl_X(S)   @>{g}>>  Y=\mycl_X(S_d) \amalg_{f''}Y'
	\end{CD}
	\]
	Here, the map $j$ is the identity on $\mycl_X(S_d)$ and equal to $f'$ on $S'$, 
	and the map $h$ is induced by the inclusion maps $\mycl_X(S_d) \to \mycl_X(S)$ and $S' \to \mycl_X(S)$.
	All four maps are surjective.
	Since $f'$ is definably proper, the map $j$ is so.
	Recall that $p$ is definably proper.
	The map $h$ is continuous.
	These facts imply that $g$ is definably proper by Lemma \ref{lem:via_homeo}.
	Since $f'$ maps $S\setminus S_d$ bijectively onto $Y'$,
	$g|_S:S \to Y$ is a bijection.
	Put $f:=g \circ \sigma:X \to Y$.
	Then $f$ is definable and surjective.
	It is also obvious that $(x_1,x_2) \in E$ if and only if $f(x_1)=f(x_2)$ because $g|_S:S \to Y$ is a bijection.
	
	We want to show that $f:X \to Y$ is a definable proper quotient of $X$ by $E$.
	The remaining task is to demonstrate that $f:X \to Y$ is continuous and definably proper.
	We first prove that $f$ is continuous.
	Take an arbitrary point $p \in X$ and an arbitrary pseudo-curve $\alpha:D \to X$ with $\myLim(\alpha)=\{p\}$.
	We have only to prove that $f(p) \in \myLim(f \circ \alpha)$ by Lemma \ref{lem:cont}.
	Consider the definable map $\gamma:=(\alpha, \sigma \circ \alpha): D \to E$.
	There exists $q \in \myLim(\sigma \circ \alpha)$ such that $(p,q) \in \myLim(\gamma) \cap E$ by Lemma \ref{lem:complete}.
	In particular, we have $f(p)=f(q)$.
	Since $\sigma \circ \alpha(D) \subseteq S$, we have $q \in \mycl_X(S)$.
	We get $g(q) \in \myLim(g \circ \sigma \circ \alpha) = \myLim( f \circ \alpha)$ by Lemma \ref{lem:cont} because $g$ is continuous.
	Since $f(p)=f(q)$, we have only to demonstrate that $g(q)=f(q)$.
	We consider two separate cases.
	\begin{enumerate}
		\item[(1)] If $q \in S$, we have $f(q)=g(q)$ because $\sigma$ is the identity on $S$.
		\item[(2)] If $ q \in \mycl_X(S) \setminus S$, we have $\sigma(q) \in S \setminus S_d$ because $B \cap S_d=\emptyset$.
		Both $q$ and $\sigma(q)$ belong to $S'$.
		We get $f'(q)=f'(\sigma(q))$ because $q$ and $\sigma(q)$ are equivalent in $E'$.
		It implies that $f(q)=g(q)$.
	\end{enumerate}
	We next demonstrate that $f$ is definably proper.
	Take a pseudo-curve $\alpha:D \to X$ and a point $y \in Y$ such that $\myLim(f \circ \alpha)= \myLim(g \circ \sigma  \circ \alpha)=\{y\}$.
	We have only to prove that $\myLim(\alpha) \cap X$ is not empty by Lemma \ref{lem:proper_eq}.
	Since $g$ is definably proper, $\myLim(\sigma  \circ \alpha)$ is not empty and contained in $X$ by Lemma \ref{lem:proper_eq}.
	By Lemma \ref{lem:complete}, $\myLim(\alpha) \cap X$ is also nonempty.
	We have finished the proof.
\end{proof}

\begin{remark}\label{rem:comment}
	Property (f) of Definition \ref{def:extended_rank} is restrictive and we do not know a structure enjoying property (f) other than d-minimal structures.
	Property (f) is only used to prove the existence of a definable proper quotient $X \to X/E$ when $\myedim_n(X)$ is the smallest element in $\mathcal E_n$ in  the proof of Theorem \ref{thm:quotient-mae}.
	It implies that, if we can prove the existence of a definable proper quotient when $\myedim_n(X)$ is the smallest element in $\mathcal E_n$ without using property (f), Theorem \ref{thm:quotient-mae} holds even when $\myedim_n$ does not possess property (f).
	Since property (f) is not used in Section \ref{sec:def_quo} and Section \ref{sec:proper_action}, the assertions in these sections hold without assuming property (f) once Theorem \ref{thm:quotient-mae} is proven without using property (f).
\end{remark}

\section{Definable quotients}\label{sec:def_quo}
We considered the case in which the definable equivalence relation is definably proper in the previous section.
This assumption seems to be too restrictive.
We relax it in this section.
We first prove the following key lemma:

\begin{lemma}\label{lem:basic}
	Consider a definably complete expansion of an ordered field $\mathcal F=(F,<,+,\cdot,0,1,\ldots)$ enjoying the definable choice property and admitting a good $(0+)$-pseudo-filter $\mathfrak D$.
	Let $X \subseteq F^m$ and $Y \subseteq F^n$ be definable sets.
	Let $Z$ be a definable subset of $Y \times X$ satisfying conditions (1) through (4) described below.
	Here, $p_1:Z \to Y$ and $p_2:Z\to X$  denote the restrictions of canonical projections to $Z$.
	\begin{enumerate}
		\item[(1)] The map $p_1$ is definably identifying.
		\item[(2)] For any $y_1,y_2 \in Y$, we have either $p_2(p_1^{-1}(y_1))  = p_2(p_1^{-1}(y_2))$ or $p_2(p_1^{-1}(y_1))  \cap p_2(p_1^{-1}(y_2))=\emptyset$.
		\item[(3)] The definable set $X$ is closed in $F^m$ and $Z$ is closed in $Y \times X$.
		\item[(4)] Let $\beta:E\to X$ be a pseudo-curve and $x$ be a point in $X$ such that $\myLim(\beta)=\{x\}$.
		Let $y \in Y$ be a point with $(y,x) \in Z$.
		Then there exist a subset $D$ of $E$ with $D \in \mathfrak D$ and  a pseudo-curve $\alpha:D \to Z$ such that $p_2 \circ \alpha(t) = \beta(t)$ for each $t \in D$ and $(y,x) \in \myLim(\alpha)$.
	\end{enumerate}
	Then there exists a definable closed subset $K$ of $X$ such that the map $p_1|_{p_2^{-1}(K)}$ is surjective and definably proper.
\end{lemma}
\begin{proof}
	In the proof, we use the following notations to clarify the ambient space containing the given point.
	Set $d_m(x)=\max_{1 \leq i \leq m}|x_i|$ and $d_m(x,y)=\max_{1 \leq i \leq m}|x_i-y_i|$ for $x=(x_1,\ldots,x_m)$ and $y=(y_1,\ldots, y_m)$.
	We define $d_n$ and $d_{m+n}$ in the same manner.
	Replacing $X$ with $X \times \{1\} \subseteq F^{m+1}$, we may assume that $d_m(x) \geq 1$ for any $x \in X$.
	Consider the definable set
	$$L(y) :=\{x \in X\;|\; (y,x) \in Z,\ d_m(x) \leq d_m(x') \text{ for any }x' \in X\text{ with }(y,x') \in Z\}$$ for any $y \in Y$ and set $L:=\bigcup_{y \in Y} L(y)$.
	\medskip
	
	\textbf{Claim 1.}
	The restriction $p_1|_{p_2^{-1}(L)}$ is surjective.
	\begin{proof}[Proof of Claim 1]
		Let $y \in Y$ be an arbitrary point.
		Note that $p_1$ is surjective by the assumption (1).
		The definable set $p_1^{-1}(y)$ is a nonempty closed subset of $F^m$ by the assumption (3).
		Take a point $x' \in p_1^{-1}(y)$. 
		The set $\{x \in p_1^{-1}(y)\;|\; d_m(x) \leq d_m(x')\}$ is definably compact.
		The restriction of $d_m$ to this set attains its minimum at a point $c \in X$ by Lemma \ref{lem:image}.
		It is obvious that $c \in L(y)$.
		We get $(y,c) \in p_2^{-1}(L)$ and $p_1(y,c)=y$.
		It means that $p_1|_{p_2^{-1}(L)}$ is surjective.
	\end{proof}

	\textbf{Claim 2.}
	We have either $L(y_1)=L(y_2)$ or $L(y_1) \cap L(y_2)=\emptyset$ for any $y_1,y_2 \in Y$.
	\begin{proof}[Proof of Claim 2]
		The claim follows immediately from the assumption (2).
	\end{proof}
	
	\textbf{Claim 3.}
	Let $y \in Y$.
	There is no definable map $\eta:(0,\delta) \to Z \cap p_2^{-1}(L)$ such that $0 \neq d_n(y,p_1(\eta(t)))<t$ and $d_m(p_2(\eta(t)))>1/t$, where $\delta>0$.
	\begin{proof}[Proof of Claim 3]
		Assume for contradiction that such an $\eta$ exists.
		Let $E$ be an element of $\mathfrak D$.
		Let $\eta'$ be the restriction of $\eta$ to $E \cap (0,\delta)$.
		We may assume that $\eta'$ is continuous by property (c) in Definition \ref{def:filter} by replacing $E$ if necessary.
		It is obvious that $\myLim(p_1 \circ \eta')=\{y\}$.
		By assumption (1) and Lemma \ref{lem:identifying_eq}, there exist a pseudo-curve $\alpha:D \to Z$ and a point $z \in  Z$ such that $p_1 \circ \alpha(D) \subseteq  p_1 \circ \eta'(E)$, $\myLim(\alpha)=\{z\}$ and $p_1(z) \in \myLim(p_1 \circ \eta')$.
		It implies that $y=p_1(z)$.
		We reduce to the case in which $\myLim(p_1 \circ \alpha)=\{y\}$.
		Since $p_1$ is continuous, we have $y=p_1(z) \in \myLim(p_1 \circ \alpha)$ by Lemma \ref{lem:cont}.
		In particular, $\myLim(p_1 \circ \alpha)$ is not empty.
		By replacing an appropriate subset of $D$ belonging $\mathcal D$ with $D$, we may assume that $\myLim(p_1 \circ \alpha)$ is a singleton by property (c) in Definition \ref{def:filter}.
		Even after the replacement, we have $\myLim(\alpha)=\{z\}$ by Lemma \ref{lem:basic_o+2}(3) and $\myLim(p_1 \circ \alpha)=\{y\}$ by Lemma \ref{lem:cont}.
		Therefore, we may assume that $\myLim(p_1 \circ \alpha)=\{y\}$.
		
		By the definable choice property, there exists a definable map $\zeta:D \to E$ such that $p_1 \circ \alpha (t)=p_1 \circ \eta' \circ \zeta(t)$.
		The union $E \cup \{0\}$ is definably compact by the definition of $\mathfrak D$.
		The limit set $\myLim(\zeta)$ is not empty and contained in $E \cup \{0\}$ by Lemma \ref{lem:basic_o+1}(2).
		We may assume that $\zeta$ is a pseudo-curve such that $\myLim(\zeta)$ is singleton by property (c) in Definition \ref{def:filter} by replacing a subset of $D$ belonging to $\mathfrak D$.
		We still have $\myLim(\alpha)=\{z\}$ and $\myLim(p_1 \circ \alpha)=\{y\}$ by Lemma \ref{lem:basic_o+2}(3) even after the replacement.
		We want to show that $\myLim(\zeta)=\{0\}$.
		Assume for contradiction that $\myLim(\zeta)=\{s\}$ for some $s \in E$.
		We have $p_1 \circ \eta'(s) \in \myLim(p_1 \circ \eta' \circ \zeta)=\myLim(p_1 \circ \alpha)=\{y\}$ by Lemma \ref{lem:cont} because $p_1 \circ \eta' $ is continuous.
		We have $p_1(\eta'(s))=y$, which contradicts the assumption that $d_n(y,p_1(\eta'(t))) \neq 0$ for each $t \in E$. 
		We have demonstrated that $\myLim(\zeta)=\{0\}$.
		
		Let us consider the definable functions $q_1,q_2:D \to \{t \in F\;|\; t>0\}$ given by $q_1(t)=1/d_m(p_2(\eta'(\zeta(t))))$ and $q_2(t)= 1/d_m(p_2(\alpha(t)))$ for each $t \in D$.
		Since $0<q_1(t)< \zeta(t)$ by the assumption and $\myLim(\zeta)=\{0\}$, we have $\myLim(q_1)=\{0\}$.
		By the definition of the set $L$, we have $d_m(p_2(\alpha(t))) \geq d_m(p_2(\eta'(\zeta(t))))$ for any $t \in D$.
		This inequality implies that $q_2 \leq q_1$.
		We get $\myLim(q_2)=\{0\}$.
		On the other hand, we have $0 \neq 1/ d_m(p_2(z)) \in \myLim(q_2)$ by Lemma \ref{lem:cont} because the map given by $t \mapsto 1/ d_m(p_2(t))$ is continuous. 
		It is a contradiction.
	\end{proof}
	
	We next define the definable function $\rho:Y \to F$ by $$\rho(y)=\inf_{\varepsilon>0}\sup\{d_m(x)\;|\exists y' \in Y\ \;x \in L(y') \text{ and } d_n(y,y') < \varepsilon\}.$$
	For the sake of well-definedness of the map $\rho$, we need to prove that $\rho(y)<\infty$.
	\medskip
	
	\textbf{Claim 4.} $\rho(y)<\infty$ for each $y \in Y$.
	\begin{proof}[Proof of Claim 4]
		Assume for contradiction that $\rho(y)=\infty$.
		Set $c=d_m(x)>0$ with $x \in L(y)$.
		It is independent of the choice of $x \in L(y)$ by the definition of $L(y)$.
		The equality  $\rho(y)=\infty$ implies that, for any $\varepsilon>0$ with $\varepsilon<1/c$, there exist $y' \in Y$ and $x \in L(y')$ such that $0 \neq d_n(y,y') < \varepsilon$ and $d_m(x) > 1/\varepsilon$.
		By the definable choice property, we can construct a definable map $\eta:(0,1/c) \to Z \cap p_2^{-1}(L)$ such that $d_n(y,p_1(\eta(t)))<t$ and $d_m(p_2(\eta(t)))>1/t$.
		In particular, we get $d_m(p_2(\eta(t)))>c=d_m(x)$.
		The definition of $x$ implies that  $d_n(y,p_1(\eta(t))) \neq 0$.
		They contradict Claim 3.
	\end{proof}
	
	By Claim 4, $\rho$ is a well-defined definable function on $Y$.
	Since $d_m(x) \geq 1$ for all $x \in X$, we have $\rho(y)>0$ for each $y \in Y$.
	We next show that $\rho$ is locally bounded.
	\medskip
	
	\textbf{Claim 5.} The function $\rho$ is locally bounded.
	\begin{proof}[Proof of Claim 5]
		Consider the map $\tau:Y \to F$ given by $$\tau(y)=\inf_{\delta>0}\sup\{\rho(y')\;|\; y' \in Y, d_n(y,y')<\delta\}.$$
		If $\tau$ is a well-defined map, for any $y \in Y$, there exist $\varepsilon>0$ and $\delta>0$ such that $\sup\{\rho(y')\;|\; y' \in Y, d_n(y,y')<\delta\}<\tau(y)+\varepsilon$.
		It implies that $\rho$ is locally bounded.
		Therefore, we have only to show that $\tau(y)<\infty$ for each $y \in Y$.
		
		Assume for contradiction that $\tau(y)=\infty$ for some $y \in Y$.
		Set $c=\rho(y)$.
		For any $\varepsilon>0$ with $\varepsilon<1/c$, there exists $y' \in Y$ such that $0 \neq d_n(y,y') <\varepsilon/2$ and $\rho(y')>1/\varepsilon$.
		Using the definable choice property, we construct a definable map $\eta:(0,1/c) \to Y$ such that $0 \neq d_n(y,\eta(t))<t/2$ and $\rho(\eta(t))>1/t$ for all $0<t<1/c$.
		Using the definable choice property again together with the definition of the function $\rho$, we get a definable map $h:(0,1/c) \to Z \cap p_2^{-1}(L)$ such that $0 \neq d_n(\eta(t),p_1(h(t)))<t/4$ and $d_m(p_2(h(t)))>1/t$.
		We have $d_n(y,p_1(h(t))) \leq d_n(y,\eta(t))+d_n(\eta(t),p_1(h(t)))<t/2+t/4<t$ and $d_m(p_2(h(t)))>1/t$.
		We also have $d_n(y,p_1(h(t))) \neq 0$ because $d_m(p_2(h(t)))>c=\rho(y)$.
		They contradict Claim 3.
	\end{proof}
	
	We set $K=\mycl(L)$.
	We demonstrate that $K$ satisfies the condition of the lemma.
	It is obvious that the restriction $p_1|_{p_2^{-1}(K)}$ is surjective because  $p_1|_{p_2^{-1}(L)}$ is so by Claim 1.
	We have only to demonstrate that  $p_1|_{p_2^{-1}(K)}$ is definably proper.
	For that purpose, we want to show that $d_m(x) \leq 3\rho(y)$ whenever $(y,x) \in p_2^{-1}(K)$.
	We first consider the case in which $x \in L$, namely $(y,x) \in p_2^{-1}(L)$.
	We have $x \in L(y)$ by Claim 2.
	The inequality $d_m(x) \leq 3\rho(y)$ is obvious by the definitions of $L(y)$ and $\rho$.
	
	We consider the remaining case.
	Fix an arbitrary $(\widetilde{y},\widetilde{x}) \in p_2^{-1}(K \setminus L)$.
	We prove that $d_m(\widetilde{x}) \leq 3\rho(\widetilde{y})$.
	Fix a sufficiently small $\varepsilon'>0$ with $\varepsilon'<\rho(\widetilde{y})$.
	The definition of $\rho$ asserts that there exists $\varepsilon>0$ such that $$(*):\ \sup\{d_m(x)\;|\; \exists y \in Y, \ x \in L(y) \text{ and } d_n(\widetilde{y},y)<\varepsilon\}<\rho(\widetilde{y})+\varepsilon'$$
	By Proposition \ref{prop:curve_selection}, there exists a pseudo-curve $\beta:E \to L$ with $\myLim(\beta)=\{\widetilde{x}\}$.
	 By the proof of Proposition \ref{prop:curve_selection}, we may further assume that $$d_m(\widetilde{x},\beta(t))<\rho(\widetilde{y})$$ when $t \in E$.
	The assumption (4) asserts that there exist a subset $D$ of $E$ belonging to $\mathfrak D$ and a pseudo-curve $\alpha:D \to Z$ such that $p_2 \circ \alpha(t) =\beta(t)$ for each $t \in D$ and $(\widetilde{y},\widetilde{x}) \in \myLim(\alpha)$.
	We can take $\widetilde{t} \in D$ with $0< \widetilde{t}  <\delta$ satisfying the inequality  $d_{m+n}(\alpha(\widetilde{t} ),(\widetilde{y},\widetilde{x}))<\varepsilon$ because $(\widetilde{y},\widetilde{x}) \in \myLim(\alpha)$.
	In particular, we get $d_n(p_1(\alpha(\widetilde{t})),\widetilde{y})<\varepsilon$.
	Using the inequality $(*)$, we get $$d_m(\beta(\widetilde{t}))<\rho(\widetilde{y})+\varepsilon'$$ because $\beta(\widetilde{t})\in L(p_1(\alpha(\widetilde{t})))$.
	We have $d_m(\widetilde{x}) \leq d_m(\beta(\widetilde{t}))+d_m(\beta(\widetilde{t}),\widetilde{x})<2\rho(\widetilde{y})+\varepsilon'< 3\rho(\widetilde{y})$.
	We have demonstrated the inequality $d_m(\widetilde{x}) \leq 3\rho(\widetilde{y})$.
	
	Let $C$ be a definably compact subset of $Y$.
	Since $\rho|_C$ is locally bounded by Claim 5, it is bounded by a positive element $N$ in $F$ by Lemma \ref{lem:local_bound}.
	It implies that $d_m(x) \leq 3N$ when $(y,x) \in p_1^{-1}(C) \cap p_2^{-1}(K)$.
	We have $(p_1|_{p_2^{-1}(K)})^{-1}(C)=p_1^{-1}(C) \cap p_2^{-1}(K) \subseteq C \times [-3N,3N]^m$.
	It implies that the set $(p_1|_{p_2^{-1}(K)})^{-1}(C)$ is bounded.
	Since $K$ is closed in $X$ and $Z$ is closed in $Y \times X$ by the assumption (3), $(p_1|_{p_2^{-1}(K)})^{-1}(C)=p_1^{-1}(C) \cap p_2^{-1}(K)=(C \times K) \cap Z$ is also closed in $C \times K$.
	Since $C \times K$ is closed in $F^{m+n}$, $(p_1|_{p_2^{-1}(K)})^{-1}(C)$ is also closed in $F^{m+n}$.
	We have shown that $(p_1|_{p_2^{-1}(K)})^{-1}(C)$ is definably compact.
	It means that $p_1|_{p_2^{-1}(K)}$ is definably proper.
\end{proof}

Using the above key lemma, we obtain the equivalence condition for a definable map to be definably identifying.
\begin{lemma}\label{lem:identifying}
	Let $\mathcal F=(F,<,+,\cdot,0,1,\ldots)$ and $\mathfrak D$ be as in Lemma \ref{lem:basic}.
	Let $X$ and $Y$ be definable sets and $f:X \to Y$ be a definable continuous map.
	We further assume that $X$ is closed.
	The following are equivalent:
	\begin{enumerate}
		\item[(i)] The map $f$ is definably identifying.
		\item[(ii)] There exists a definable closed subset $K$ of $X$ such that the restriction $f|_K$ of $f$ to $K$ is surjective and definably proper.
	\end{enumerate}
\end{lemma}
\begin{proof}
	We first prove that (ii) implies (i).
	Let $\beta:E \to Y$ be a pseudo-curve such that $\myLim(\beta)$ is a singleton and contained in $Y$.
	By Lemma \ref{lem:identifying_eq}, we have only to prove that there exists $D \in \mathfrak D$, a definable map $\alpha:D \to X$ and a point $x \in \myLim(\alpha) \cap X$ such that $f \circ \alpha(D) \subseteq \beta(E)$ and $f(x) \in \myLim(\beta)$.
	By the definable choice property, there exists a definable map $\gamma: E \to K$ such that $f \circ \gamma(t)=\beta(t)$ for each $t \in E$.
	Take an appropriate  subset $D$ of $E$ belonging to $\mathcal D$.
	The restriction $\alpha$ of $\gamma$ to $D$ is a pseudo-curve by property (c) in Definition \ref{def:filter}.
	We have $\myLim(f \circ \alpha)=\myLim(\beta)$ by Lemma \ref{lem:basic_o+2}(3) because $f \circ \gamma=\beta$ and $\myLim(\beta)$ is a singleton.
	The set $\myLim(\alpha)$ is not empty and contained in $K$ by applying Lemma \ref{lem:proper_eq} to $f|_K$.
	Take a point $x \in \myLim(\alpha)$, then we have $f(x) \in \myLim(f \circ \alpha) = \myLim(\beta)$ by Lemma \ref{lem:cont} because $f$ is continuous.
	We have proven the implication (ii) $\Rightarrow$ (i).
	
	We next consider the opposite implication.
	We want to apply Lemma \ref{lem:basic} to the definable set $Z=\{(f(x),x) \in Y \times X\}$.
	Let $p_1,p_2$ be as in Lemma \ref{lem:basic}.
	We check that conditions (1) through (4) in Lemma \ref{lem:basic} are all satisfied.
	Condition (1) immediately follows from the assumption that $f$ is definably identifying.
	Condition (2) is obviously satisfied because $p_2(p_1^{-1}(y))=f^{-1}(y)$ in this case.
	Condition (3) holds because $f$ is a continuous map defined on a closed set and $Z$ is the permuted graph of $f$.
	Condition (4) is obvious from the fact that $Z$ is the permuted image of the graph of a continuous map by Lemma \ref{lem:cont}.
	The definable map $p_1|_{p_2^{-1}(K)}$ is surjective and definably identifying for some closed definable subset $K$ of $X$ by Lemma \ref{lem:basic}.
	The image of $f|_K$ coincides with that of $p_1|_{p_2^{-1}(K)}$, and $f|_K$ is surjective.
	Let $C$ be a definably compact subset of $Y$.
	Since $p_1|_{p_2^{-1}(K)}$ is definably proper, $(p_1|_{p_2^{-1}(K)})^{-1}(C)$ is  definably compact.
	The inverse image $(f|_K)^{-1}(C)= p_2((p_1|_{p_2^{-1}(K)})^{-1}(C))$ is also definably compact by Lemma \ref{lem:image}.
	It means that $f|_K$ is definably proper.
\end{proof}
%

%
%
%

\begin{lemma}\label{lem:geo_quo}
	Let $\mathcal F=(F,<,+,\cdot,0,1,\ldots)$ and $\mathfrak D$ be as in Lemma \ref{lem:basic}.
	Let $X$ be a definable closed set and $E$ be a definable equivalence relation on $X$ which is closed in $X \times X$.
	Let $K$ be a definable closed subset of $X$ and set $E_K=E \cap (K \times K)$.
	Let $p_i:E \to X$ be the restriction of the canonical projection $X \times X \to X$ onto the $i$-th factor to $E$ for $i=1,2$.
	Assume that the restriction $p_1|_{p_2^{-1}(K)}$ is definably identifying.
	There exists a definable quotient of $X$ by $E$ if and only if a definable quotient of $K$ by $E_K$ exists.
	Furthermore, if $f:X \to Y$ is a definable quotient of $X$ by $E$, the restriction $f|_K:K \to Y$ is also a definable quotient of $K$ by $E_K$.
\end{lemma}
\begin{proof}
	Set $q_i=p_i|_{p_2^{-1}(K)}$.
	We first assume that there exists a definable quotient of $X$ by $E$.
	Let $f:X \to Y$ be the definable quotient.
	Consider the following commutative diagram:
	\[
	\begin{CD}
		p_2^{-1}(K) @>{q_2}>> K\\
		@V{q_1}VV @VV{f|_K}V\\
		X @>{f}>> Y
	\end{CD}
	\]
	Both $q_1$ and $f$ are definably identifying by the assumption.
	The restriction $f|_K$ is also definably identifying by Lemma \ref{lem:identifying_basic}.
	It implies that $f|_K$ is a definable quotient of $K$ by $E_K$.
	
	We next consider the case in which a definable quotient of $K$ by $E_K$ is given.
	Let $g:K \to Y$ be the definable quotient.
	Let $f:X \to Y$ be the definable map given by $f(x)=g(\widetilde{x})$, where $\widetilde{x}$ is an element in $K$ with $(x,\widetilde{x}) \in E$.
	Such $\widetilde{x}$ always exists because $q_1$ is surjective.
	In addition, $f$ is independent of the choice of $\widetilde{x}$ because $g$ is a definable quotient.
	We want to show that $f$ is continuous. 
	Let $\alpha:D \to X$ be a pseudo-curve and $p \in X$ with $\myLim(\alpha)=\{p\}$.
	We have only to show that $f(p) \in \myLim(f \circ \alpha)$ by Lemma \ref{lem:cont}.
	Set $U=(f \circ \alpha)^{-1}(f(p))$ and $V=D \setminus U$.
	By property (b) in Definition \ref{def:filter}, at least one of $U$ and $V$ contains an element $D'$ of $\mathfrak D$.
	If $D'$ is contained in $U$, it is obvious that $f(p) \in \myLim(f \circ \alpha)$.
	We have proven the claim in this case.
	Otherwise, we may assume that $f(p) \notin f \circ \alpha(D)$ by replacing $D$ with $D'$ by Lemma \ref{lem:basic_o+2}(3).
	Since $q_1$ is definably identifying, there exist $D'' \in \mathfrak D$, a definable map $\beta:D'' \to E \cap (X \times K)$ and a point $q \in K$ such that $p_1 \circ \beta(D'') \subseteq \alpha(D)$ and $(p,q) \in \myLim(\beta) \cap E$ by Lemma \ref{lem:identifying_eq}.
	The inclusion $p_1 \circ \beta(D'') \subseteq \alpha(D)$ implies the inclusion $g \circ p_2 \circ \beta(D'') \subseteq f \circ \alpha(D)$.
	By the definition of $f$ and $g$, we have $f(p)=g(q)$.
	Since $g \circ p_2$ is continuous, we have $g(q) \in \myLim(g \circ p_2 \circ \beta)$  by Lemma \ref{lem:cont}.
	By Lemma \ref{lem:basic_o+2}(1), we have $g \circ p_2 \circ \beta(D'') \cup \myLim(g \circ p_2 \circ \beta) = \mycl(g \circ p_2 \circ \beta(D'')) \subseteq \mycl( f \circ \alpha(D))= f \circ \alpha(D) \cup \myLim( f \circ \alpha(D))$.
	In particular, we get $f(p) \in f \circ \alpha(D) \cup \myLim( f \circ \alpha(D))$ because $f(p)=g(q) \in \myLim(g \circ p_2 \circ \beta)$.
	It implies that $f(p) \in \myLim(f(\alpha))$ because $f(p) \notin f \circ \alpha(D)$.
	We have demonstrated that $f$ is continuous.
	
	As an intermediate step, we show that the map $q_2$ is definably identifying.
	Let $\gamma:G \to K$ be a pseudo-curve and $v \in K$ such that $\myLim(\gamma)=\{v\}$.
	Since $E$ is a definable equivalence relation, we have $(\gamma(t),\gamma(t)) \in E$ for any $t \in G$ and $(v,v) \in \myLim((\gamma,\gamma)) \cap E_K$.
	It means that $q_2$ is definably identifying by Lemma \ref{lem:identifying_eq}.
	
	Consider the following commutative diagram:
	\[
	\begin{CD}
		p_2^{-1}(K) @>{q_2}>> K\\
		@V{q_1}VV @VV{g}V\\
		X @>{f}>> Y
	\end{CD}
	\]
	This diagram implies that $f$ is definably identifying by Lemma \ref{lem:identifying_basic} because $g$ and $q_2$ are definably identifying.
\end{proof}

The following theorem is the main theorem of this section:

\begin{theorem}\label{thm:quotinet}
	Consider a definably complete expansion of an ordered field $\mathcal F=(F,<,+,\cdot,0,1,\ldots)$ enjoying the definable choice property and admitting a good $(0+)$-pseudo-filter $\mathfrak D$.
	Assume further that $\mathcal F$ has a good extended rank function $\myedim$.
	Let $X$ be a locally closed definable set and $E$ be a definable equivalence relation on $X$ which is closed in $X \times X$.
	The following are equivalent:
	\begin{enumerate}
		\item[(1)] There exists a definable quotient of $X$ by $E$.
		\item[(2)] Let $p_i$ be the restriction of the canonical projection $X \times X \to X$  onto the $i$-th coordinate to $E$ for $i=1,2$.
		There exists a definable closed subset $K$ of $X$ such that the restriction $p_1|_{p_2^{-1}(K)}$ of $p_1$ to $p_2^{-1}(K)$ is definably proper and surjective.
	\end{enumerate}
\end{theorem}
\begin{proof}
	Let $F^m$ be the ambient space of $X$.
		We first reduce to the case in which $X$ is closed.
		We have nothing to do when $X$ is closed.
		We consider the case in which $X$ is not closed.
		Since $X$ is locally closed, its frontier $\partial X:=\mycl(X) \setminus X$ is closed.
			Let $P$ be the definable function $P:F^m \to F$ given by $P(x)=\inf\{|x'-x|\;|\; x' \in \partial X\}$.
		It is obvious that $P^{-1}(0)=\partial X$ and $P$ is continuous.
		Consider the definable map $\varphi:X \to F^{n+1}$ given by $\varphi(x)=(x,1/P(x))$.
		It is obvious that $\varphi$ is a homeomorphism onto its image and its image is closed.
		It is also obvious that a definable set $C$ is definably compact if and only if $\varphi^{-1}(C)$ is definably compact by Lemma \ref{lem:image}.
		Therefore, we can reduce to the case in which $X$ is closed.
	
	We first demonstrate the implication $(1) \Rightarrow (2)$.
	Let $f:X \to Y$ be a definable quotient of $X$ by $E$.
	Since $f$ is definably identifying, there exists a definable closed subset $K$ of $X$ such that the restriction $f|_K:K \to Y$ is definably proper and surjective by Lemma \ref{lem:identifying}.
	Set $q_i=p_i|_{p_2^{-1}(K)}$ for $i=1,2$.
	Consider the following commutative diagram:
	\[
	\begin{CD}
		p_2^{-1}(K) @>{q_2}>> K\\
		@V{q_1}VV @VV{f|_K}V\\
		X @>{f}>> Y
	\end{CD}
	\]
	The map $q_1$ is surjective.
	In fact, by the definition of $K$, for any $x \in X$, there exists $x' \in K$ with $f(x)=f(x')$.
	The point $(x,x')$ is contained in $p_2^{-1}(K)$ and $q_1(x,x')=x$.
	The remaining task is to demonstrate that $q_1$ is definably proper.
	Let $C$ be a nonempty definably compact subset of $X$.
	The image $f(C)$ is definably compact by Lemma \ref{lem:image} because $f$ is continuous.
	Since $f|_K$ is definably proper, $(f|_K)^{-1}(f(C))$ is definably compact.
	By the above commutative diagram, we get $q_1^{-1}(C) \subseteq C \times (f|_K)^{-1}(f(C))$.
	It implies that $q_1^{-1}(C)$ is bounded.
	It is obvious that $q_1^{-1}(C)$ is closed in $F^{2m}$ because $E$ is closed in $F^{2m}$.
	They imply that $q_1^{-1}(C)$ is definably compact and we have demonstrated that $q_1$ is definably proper.
	
	Our next task is to demonstrate the implication $(2) \Rightarrow (1)$.
	Consider the following commutative diagram:
	\[
	\begin{CD}
		E_K @>{\hookrightarrow}>> p_2^{-1}(K)\\
		@V{p_1|_{E_K}}VV @VV{q_1}V\\
		K @>{\hookrightarrow}>> X
	\end{CD}
	\]
	Here, $E_K$ denotes the set $E \cap (K \times K)$.
	We want to show that $p_1|_{E_K}$ is definably proper.
	Let $C$ be a definably compact subset of $K$.
	Since $q_1$ is definably proper, $q_1^{-1}(C)$ is definably compact.
	We obviously have $(p_1|_{E_K})^{-1}(C)=(C \times K) \cap E = q_1^{-1}(C)$.
	Therefore, $(p_1|_{E_K})^{-1}(C)$ is definably compact, and $p_1|_{E_K}$ is definably proper.
	By Theorem \ref{thm:quotient-mae}, a definable proper quotient of $K$ by $E_K$ exists.
	It is also a definable quotient by the the definitions of quotients and Corollary \ref{cor:proper_identifying}.
	A definable quotient of $X$ by $E$ exists by Lemma \ref{lem:geo_quo}.
\end{proof}

\begin{remark}\label{rem:locally_omin}
	A definably complete locally o-minimal expansion of an ordered field satisfies the assumptions in Theorem \ref{thm:quotinet}; that is, it enjoys the definable choice property,  admits a good $(0+)$-pseudo-filter $\mathfrak D$, and has a good extended rank function $\myedim$.
\end{remark}
\begin{proof}
	The definable choice property is given in \cite[Lemma 2.8]{Fuji_decomp}.
	The other properties follow from Example \ref{ex:local1} and Example \ref{ex:local2}.
\end{proof}

\section{Application to definably proper action}\label{sec:proper_action}

We consider an action of a definable group on a definable set.

\begin{definition}
	A definable set $G$ is a \textit{definable group}
	if $G$ is a group and the group operations $G \times G \ni (g,h) \mapsto gh \in G$ and $G  \ni g\mapsto g^{-1} \in G$ are definable continuous maps.
	
	A \textit{definable $G$-set} is a pair $(X, \phi)$ consisting of a definable set $X$ and a group action $\phi:G \times X \to X$ which is a definable continuous map. 
	We simply write $X$ instead of $(X, \phi)$ and $gx$ instead of $\phi(g,x)$.
	We say that a $G$-invariant definable subset of $X$ is a \textit{definable $G$-subset} of $X$.
\end{definition}

\begin{theorem}\label{thm:definable_action}
	Consider a definably complete expansion of an ordered field $\mathcal F=(F,<,+,\cdot,0,1,\ldots)$ enjoying the definable choice property and admitting a good $(0+)$-pseudo-filter $\mathfrak D$.
	Assume further that $\mathcal F$ has a good extended rank function $\myedim$.
	Let $G$ be a definable group and $X$ be a locally closed definable $G$-set.
	Assume further that $Z=\{(x,gx) \in X \times X\;|\; g \in G, x\in X \}$ is closed in $X \times X$.
	Then, there exists a definable quotient $X \to X/G$.
\end{theorem}
\begin{proof}
	Let $X$ be a definable subset of $F^m$.
	We may assume that $X$ is closed in the same manner as the proof of Theorem \ref{thm:quotinet}.
	Consider the definable equivalence relation $Z$ and apply Lemma \ref{lem:basic} to this set.
	
	We first check that conditions (1) through (4) in Lemma \ref{lem:basic} are all satisfied.
	Let $p_1,p_2:Z \to X$ be the map defined in the same way as Lemma \ref{lem:basic}.
	Satisfaction of condition (1) is easily proven.
	Let $\gamma:D \to X$ be a pseudo-curve and $p$ be a point in $X$ such that $\myLim(\gamma)=\{p\}$.
	The map $\alpha:D \to Z$ given by $\alpha(t)=(\gamma(t),\gamma(t))$ satisfies condition (ii) of Lemma \ref{lem:identifying_eq}.
	It means $p_1$ is definably identifying.
	Condition (2) is satisfied because the definable set $Z$ is a definable equivalence relation defined on $X$.
	Condition (3) immediately follows from the assumptions of the theorem.
	We finally check condition (4).
	Let $\beta:D \to X$ be a pseudo-curve and $x$ be a point in $X$ such that $\myLim(\beta)=\{x\}$.
	Take an arbitrary point $y \in X$ with $(y,x) \in Z$.
	There exists $g \in G$ with $y=gx$ by the definition of $Z$.
	Consider the pseudo-curve $\alpha:D \to Z$ given by $\alpha(t)=(g \cdot \beta(t),\beta(t))$ satisfies the requirement in (4).
	
	We now apply Lemma \ref{lem:basic} to $Z$.
	There exists a definable closed set $K$ of $X$ such that $p_1|_{p_2^{-1}(K)}$ is surjective and definably proper.
	The definable quotient exists by Theorem \ref{thm:quotinet}.
\end{proof}

Recall the definition of definably proper actions.

\begin{definition}
	Consider an expansion of a dense linear order without endpoints.
	Let $G$ be a definable group and $X$ be a definable $G$-set.
	The $G$-action on $X$ is called \textit{definably proper} if the map $G \times X \ni (g,x) \mapsto (x,gx) \in X \times X$ is a definably proper map.
\end{definition}

\begin{theorem}\label{thm:definable_action2} 
	Let $\mathcal F=(F,<,+,\cdot,0,1,\ldots)$ be as in Theorem \ref{thm:definable_action}.
	Let $G$ be a definable group and $X$ be a definable $G$-set which is locally closed.
	Assume further that the $G$-action on $X$ is definably proper.
	Then, there exists a definable quotient $X \to X/G$.
\end{theorem}
\begin{proof}
	Apply Lemma \ref{lem:proper_closed} to the definably proper map $G \times X \ni (g,x) \mapsto (x,gx) \in X \times X$.
	The equivalence relation $E=\{(x,gx) \in X \times X\;|\; g \in G,x\in X\}$ is closed in $X \times X$ because it is the image of the above definably proper map.
	The theorem follows from Theorem \ref{thm:definable_action}.
\end{proof}

\begin{proposition}\label{prop:invariant_partition}
	Consider a definably complete expansion of an ordered field $\mathcal F=(F,<,+,\cdot,0,1,\ldots)$.
	Assume further that $\mathcal F$ enjoys definable choice property and has a good extended rank function $\myedim$.
	Let $G$ be a definable group and $X$ be a definable $G$-set.
	The set $X$ is partitioned into finitely many definable $G$-subsets of $X$ which are locally closed.
\end{proposition}
\begin{proof}
	Let $F^n$ be the ambient space of $X$.
	We demonstrate the proposition by induction on $d:=\myedim_n(X)$.
	When $d$ is a smallest element of $\mathcal E_n$, Lemma \ref{lem:contained} implies that $X$ is closed in $F^n$.
	The theorem is trivial in this case.
	
	We next consider the case in which $d$ is not a smallest element of $\mathcal E_n$.
	We have $\myedim_n(X)=\myedim_n(\mycl_{F^n}(X))$ by Lemma \ref{lem:contained} and property (d) in Definition \ref{def:extended_rank}.
	There is a definable subset $Y$ of $X$ such that $Y$ is open in $\mycl_{F^n}(X)$ and $\myedim_n(X \setminus Y)<d$ by property (e) in Definition \ref{def:extended_rank}.
	Set $Z=GY=\{gy\in X\;|\; g \in G, y \in Y\}$.
	It is obvious that $\myedim_n (X \setminus Z)<d$ and $Z$ is $G$-invariant.
	We prove that $Z$ is locally closed.
	Once it is proved, the proposition obviously follows from the induction hypothesis.
	
	Fix an arbitrary point $z \in Z$.
	We have only to demonstrate that there exists an open box $U$ containing the point $z$ such that $Z \cap U$ is closed in $U$.
	We get $z=gy$ for some $g \in G$ and $y \in Y$ because $Z=GY$.
	Since $Y$ is open in $\mycl_{F^n}(X)$, we can take a small open box $B$ containing the point $y$ such that $\mycl_{F^n}(X) \cap B=Y \cap B$.
	Take a bounded closed box $V$ contained $B$ which is simultaneously a neighborhood of $y$.
	We have $\mycl_{F^n}(X) \cap V = Y \cap V$.
	Therefore, $Y \cap V$ is definably compact because $\mycl_{F^n}(X) \cap V $ is so.
	We also get $X \cap V=Y \cap V$ because $Y \subseteq X \subseteq \mycl_{F^n}(X) $.
	Set $C=g(Y \cap V) \subseteq Z$.
	It is definably compact by Lemma \ref{lem:image}.
	In particular, $C \cap U$ is closed in $U$ for any open box $U$.
	Therefore, the remaining task is to show that there exists an open box $U$  containing the point $z$ such that $Z \cap U=C \cap U$.
	
	Assume for contradiction that $Z \cap U \neq C \cap U$ for any open box $U$ containing the point $z$.
	It implies that $z \in \partial_{F^n} (Z \setminus C ) $ because $z \in C$.
	By the definable choice property, we can find a definable map $\gamma:(0, \infty)\to Z \setminus C$ such that $|z-\gamma(t)|<t$ for each $t>0$.
	We have $\myLim(\gamma)=\{z\}$.
	Consider the map $g^{-1} \cdot \gamma$ given by $t \mapsto g^{-1}\gamma(t)$. 
	We have $\myLim(g^{-1} \cdot \gamma)=\{g^{-1}z\}=\{y\}$ because the multiplication by $g^{-1}$ is a definable homeomorphism.
	Since $V$ is a neighborhood of $y$ and $\myLim(g^{-1} \cdot \gamma)=\{y\}$, there exists $t >0$ with $g^{-1}\gamma(t) \in V$. 
	On the other hand, we have $g^{-1}\gamma(t) \not\in g^{-1}C=Y \cap V$ and $g^{-1}\gamma(t) \in g^{-1}X =X$ by the definition of $\gamma$.
	It implies that $g^{-1}\gamma(t) \in (X \cap V) \setminus (Y\cap V)$.
	It contradicts the equality $X \cap V=Y\cap V$.
\end{proof}

\begin{corollary}\label{cor:definable_top_locally_closed}
	Let $\mathcal F=(F,<,+,\cdot,0,1,\ldots)$ be as in Proposition \ref{prop:invariant_partition}.
	A definable group is locally closed in the ambient space.
\end{corollary}
\begin{proof}
	The multiplication in $G$ is a definable continuous $G$-action on $G$.
	Proposition \ref{prop:invariant_partition} implies that $G$ is partitioned into finitely many $G$-invariant locally closed definable sets.
	However, there are no $G$-invariant definable subset of $G$ other than $G$.
	Therefore, $G$ itself is locally closed.   
\end{proof}

\begin{corollary}
	Let $\mathcal F=(F,<,+,\cdot,0,1,\ldots)$ be as in Theorem \ref{thm:definable_action}.
	Let $G$ be a definable group and $H$ be a definable subgroup of $G$.
	Then, there exists a definable quotient $G \to G/H$.
\end{corollary}
\begin{proof}
	The definable group $G$ is locally closed in the ambient space $F^n$ by Corollary \ref{cor:definable_top_locally_closed}.
	Note that $H$ is closed in $G$ by \cite[Proposition 2.7]{P} because any definable set is constructible by Lemma \ref{lem:constructible}.
	Using this fact and Lemma \ref{lem:proper_eq}, it is easy to show that the map $H \times G \ni (h,g) \mapsto (g,hg) \in G \times G$ is definably proper.
	The proof is left to readers.
	The corollary follows from Theorem \ref{thm:definable_action2}.
\end{proof}

We assume that $X$ is locally closed in Theorem \ref{thm:definable_action2}.
We investigate the case in which $X$ is not necessarily locally closed.

\begin{lemma}\label{lem:sub_proper}
	Consider an expansion of a dense linear order without endpoints.
	Let $G$ be a definable  group and $X$ be a definable $G$-set.
	Let $Y$ be a definable $G$-subset of $X$.
	If the action of $G$ on $X$ is definably proper, then the action of $G$ on $Y$ is also definable proper.
\end{lemma}
\begin{proof}
	Apply Lemma \ref{lem:via_homeo}(iii) to the definable map $G \times X \ni (g,x) \mapsto (x,gx) \in X \times X$ and $Y \times Y$.
\end{proof}

\begin{corollary}
Let $\mathcal F=(F,<,+,\cdot,0,1,\ldots)$ be as in Theorem \ref{thm:definable_action}.
Let $G$ be a definable group and $X$ be a definable $G$-set.
Assume further that the $G$-action on $X$ is definably proper.
Then there exists a partition of $X$ into finitely many definable $G$-subsets $X_1, \ldots, X_m$ of $X$ such that a definable quotient $X_i \to X_i/G$ exists for each $1 \leq i \leq m$.	
\end{corollary}
\begin{proof}
	We get a partition of $X$ into finitely many definable $G$-subsets $X_1, \ldots, X_m$ of $X$  such that $X_i$ are locally closed for each $1 \leq i \leq m$ by Proposition \ref{prop:invariant_partition}.
	The $G$-action on $X_i$ is definably proper by Lemma \ref{lem:sub_proper}.
	The existence of definable quotients follows from Theorem \ref{thm:definable_action2}.
\end{proof}

%
%

\section{d-minimal case}\label{sec:d-minimal}

We prove that d-minimal expansions of an ordered field satisfy the assumptions of the assertions in Section \ref{sec:preliminary} through Section \ref{sec:proper_action}.
We first prove that a d-minimal structure has a good rank function.
For that purpose, we recall the definition of dimension and basic properties of dimension.
\begin{definition}[Dimension]\label{def:dim}
	Consider an expansion of a dense linear order without endpoints.
	Let $F$ be the universe.
	We consider that $F^0$ is a singleton with the trivial topology.
	Let $X$ be a nonempty definable subset of $F^n$.
	The dimension of $X$ is the maximal nonnegative integer $d$ such that $\pi(X)$ has a nonempty interior for some coordinate projection $\pi:F^n \rightarrow F^d$.
	We set $\dim(X)=-\infty$ when $X$ is an empty set.
\end{definition}

The following four lemmas are proven under more relaxed conditions, but we assume that the structure is a d-minimal expansion of an ordered field in this paper.

\begin{lemma}\label{lem:dim3}
	Consider a d-minimal expansion of an ordered field $\mathcal F=(F,<,+,\cdot,0,1,\ldots)$.
	Let $X$ be a definable subset of $F^n$ of dimension $d$ and let $\pi:F^n \to F^d$ be a coordinate projection such that $\pi(X)$ has a nonempty interior.
	Then there exists a definable dense open subset $U$ of $F^d$ such that 
	$\pi^{-1}(x) \cap X$ is of dimension zero and $\mycl(X) \cap \pi^{-1}(x) = \mycl(X \cap \pi^{-1}(x)) $ for each $x \in U$. 
\end{lemma}
\begin{proof}
	See \cite[Theorem 3.10(8)]{Fornasiero}.
\end{proof}

\begin{lemma}\label{lem:dim}
	Let $\mathcal F=(F,<,+,\cdot,0,1,\ldots)$ be as in Lemma \ref{lem:dim3}.
	The equality $\dim \mycl(X) = \dim X$ holds for each definable set $X$.
\end{lemma}
\begin{proof}
	See \cite[Theorem 3.10(5)]{Fornasiero}.
\end{proof}

\begin{lemma}\label{lem:dim2}
	Let $\mathcal F=(F,<,+,\cdot,0,1,\ldots)$ be as in Lemma \ref{lem:dim3}.
	The equality $\dim (A \cup B) = \max\{\dim A, \dim B\}$ holds for definable subsets $A$ and $B$ of $F^n$.
\end{lemma}
\begin{proof}
	See \cite[Theorem 3.10(12)]{Fornasiero}.
\end{proof}

\begin{lemma}\label{lem:dim4}
	Let $\mathcal F=(F,<,+,\cdot,0,1,\ldots)$ be as in Lemma \ref{lem:dim3}.
	Let $A$ be a definable subset of $F^{n+m}$ and $\pi:F^{n+m} \to F^n$ be the coordinate projection onto the first $n$ coordinates.
	The definable set $\{x \in F^n \;|\; (\partial A) \cap \pi^{-1}(x) \neq \partial (A \cap \pi^{-1}(x)) \}$ has an empty interior.
\end{lemma}
\begin{proof}
	See \cite[Theorem 3.10(8)]{Fornasiero}.
\end{proof}

We also need the following lemmas:
\begin{lemma}\label{lem:dimzeo}
Consider a d-minimal expansion of an ordered field.
The image of a definable set of dimension zero under a definable map is of dimension zero.
\end{lemma}
\begin{proof}
	It follows from \cite[Lemma 4.5]{Fornasiero}, \cite[Corollary 1.5]{vdD2} and  \cite{Miller-choice}.
\end{proof}

\begin{lemma}\label{lem:dimzero2}
Consider a d-minimal expansion of an ordered field.
Let $F$ be the universe and $X$ be a definable subset of $F^n$ of dimension zero.
For any $x \in F^{n-1}$, the fiber $X_x:=\{y \in F\;|\; (x,y) \in X\}$ is either empty or a union of finitely many discrete sets.
\end{lemma}
\begin{proof}
	The set $\{x \in F^{n-1}\;|\;X_x \text{ has a nonempty interior}\}$ is empty by \cite{Miller-choice} and \cite[Lemma 4.5(Dim 4)]{Fornasiero}.
	The lemma follows from this fact and d-minimality.
\end{proof}

We next recall the definition of rank.
\begin{definition}[\cite{FM}]\label{def:lpt}
	We denote the set of isolated points in $S$ by $\myIso(S)$ for any topological space $S$.
	We set $\myLpt(S):=S \setminus \myIso(S)$.
	In other word, a point $x \in S$ belongs to $\myLpt(S)$ if and only if $x \in \mycl_S(S \setminus \{x\})$.
	
	Let $X$ be a nonempty closed subset of a topological space $S$.
	We set $X\langle 0 \rangle=X$ and, for any $m>0$, we set $X \langle m \rangle = \myLpt(X \langle m-1 \rangle)$.
	We say that $\myrank(X)=m$ if $X \langle m \rangle=\emptyset$ and $X\langle m-1 \rangle \neq \emptyset$.
	We say that $\myrank X = \infty$ when $X \langle m \rangle \neq \emptyset$ for every natural number $m$.
	We set $\myrank(Y):=\myrank(\mycl(Y))$ when $Y$ is a nonempty subset of $S$ which is not necessarily closed.
\end{definition}

\begin{lemma}\label{lem:very_basic}
	Let $\mathcal F=(F,<,\ldots)$ be an expansion of a dense linear order without endpoints.
	For a definable closed subset $A$ of $F$ with empty interior, $\myrank(A)=k$ if and only if $k$ is the least number of discrete sets whose union is $A$.
\end{lemma}
\begin{proof}
	See \cite[1.3]{FM}.
	We can prove the lemma in the same manner as it.
	We omit the details.
\end{proof}

\begin{corollary}\label{cor:rank_frontier}
	Let $\mathcal F=(F,<,\ldots)$ be as in Lemma \ref{lem:very_basic}.
	We have $\myrank(\partial A)<\myrank(A)$ for each nonempty definable subset $A$ of $F$ with $\myint_F(A)=\emptyset$.
\end{corollary}
\begin{proof}
	It is obvious from Lemma \ref{lem:very_basic} because $\myIso(A)$ is open in $\mycl(A)$.
\end{proof}

\begin{lemma}\label{lem:very_basic2}
	Let $S$ and $T$ be subsets of a topological space.
	We have $\myLpt(S \cup T)= \myLpt(S) \cup \myLpt(T)$.
\end{lemma}
\begin{proof}
	When $x \notin \myLpt(S) \cup \myLpt(T)$, there exist open neighborhoods $U_S$ and $U_T$ of $x$ such that $U_S \cap S \setminus \{x\} = \emptyset$ and $U_T \cap T \setminus \{x\} = \emptyset$.
	We set $U:=U_S \cap U_T$, then $U$ is an open neighborhood of $x$ and $U \cap (S \cup T) \setminus \{x\} = \emptyset$.
	It means that $x \notin \myLpt(S \cup T)$.
	We have proven the inclusion $\myLpt(S \cup T) \subseteq \myLpt(S) \cup \myLpt(T)$.
	
	We next prove the opposite inclusion.
	Let $x \in \myLpt(S) \cup \myLpt(T)$.
	We may assume that $x \in \myLpt(S)$ without loss of generality.
	It implies that $x \in \mycl(S \setminus \{x\}) \subseteq \mycl(S \cup T \setminus \{x\})$.
	It means that $x \in \myLpt(S \cup T)$.
	We have proven the inclusion $\myLpt(S) \cup \myLpt(T) \subseteq \myLpt(S \cup T)$.
\end{proof}

\begin{lemma}\label{lem:very_basic3}
	Let $\mathcal F=(F,<,\ldots)$ be as in Lemma \ref{lem:very_basic}.
	Let $A_1, \ldots, A_m$ be a finite family of definable subsets of $F^n$.
	We have $\myrank(\bigcup_{i=1}^m A_i) = \max_{1 \leq i \leq m}  \myrank(A_i)$.
\end{lemma}
\begin{proof}
	It was proved in \cite[1.4]{FM} when $n=1$.
	The proof is almost the same for general $n$.
	We may assume that $A_i$ are closed without loss of generality.
	The equality $\myrank(\bigcup_{i=1}^m A_i) = \max_{1 \leq i \leq m}  \myrank(A_i)$ follows from Lemma \ref{lem:very_basic2} by easy induction.
\end{proof}

\begin{lemma}\label{lem:dmin_countable}
	Consider a d-minimal expansion of an ordered field.
	A nonempty definable set $X$ of dimension zero has an isolated point and its rank is finite.
\end{lemma}
\begin{proof}
	We may assume that  $X$ is closed without loss of generality because an isolated point of $\mycl(X)$ always belongs to $X$ as an isolated point.
	Let $X$ be a definable subset of $F^n$ of dimension zero, where $F$ is the universe.
	We show the lemma by induction on $n$.
	Consider the case in which $n=1$.
	The set $X$ does not contain a nonempty interval because $X$ is of dimension zero.
	It means that $X$ is the union of finitely many discrete sets by d-minimality.
	The lemma follows from Lemma \ref{lem:very_basic} and d-minimality.
	
	We next consider the case in which $n>1$.
	Let $\pi_1:F^n \to F^{n-1}$ and $\pi_2:F^n \to F$ be the projections forgetting the last coordinate and onto the last coordinate, respectively.
	\medskip
	
	\textbf{Claim. } A point $x \in X$ is isolated in $X$ if $\pi_1(x)$ is isolated in $\pi_1(X)$ and $\pi_2(x)$ is isolated in $X_{\pi_1(x)}$, where $X_t:=\pi_2(\pi_1^{-1}(t) \cap X)$ for each $t \in F^{n-1}$. 
	\begin{proof}[Proof of Claim]
		The proof is easy. We omit it.
	\end{proof}
	
	The image $\pi_1(X)$ is of dimension zero by Lemma \ref{lem:dimzeo}.
	By the inductive hypothesis, $\myIso(\pi_1(X))$ is not empty.
	For any $t \in \myIso(\pi_1(X))$, $X_t$ is a union of finitely many discrete sets by Lemma \ref{lem:dimzero2}.
	It means that $\myIso(X_t)$ is not empty by Lemma \ref{lem:very_basic}.
	Therefore, $\myIso(X)$ is not empty by Claim.
	
	We next prove that $\myrank(X)$ is finite.
	There exists a natural number $N$ such that $\myrank(X_t) \leq N$ for any $t \in \pi_1(X)$ by Lemma \ref{lem:dimzero2} and \cite[Lemma 5.10]{Fornasiero}.
	Set $X\langle i \rangle$ as in Definition \ref{def:lpt} for each positive integer $i$.
	By Claim, we have $\pi_1^{-1}(\myIso(\pi_1(X))) \cap X \subseteq \bigcup_{i<N}\myIso(X\langle i \rangle)$.
	We immediately get that $\myrank(X)$ is finite by induction on $r=\myrank(\pi_1(X))$.
\end{proof}

We have finished the preparation.
We introduce the extended rank function $\myedim$ for d-minimal structures.

\begin{definition}[Extended rank]\label{def:extended_dim_dmin}
	Consider a d-minimal expansion of an ordered field $\mathcal F=(F,<,+,\cdot,0,1,\ldots)$.
	Let $\Pi(n,d)$ be the set of coordinate projections of $F^n$ onto $F^d$.
	Recall that $F^0$ is a singleton.
	We consider that $\Pi(n,0)$ is a singleton whose element is a trivial map onto $F^0$.
	Since $\Pi(n,d)$ is a finite set, we can define a linear order on it. 
	We denote it by $<_{\Pi(n,d)}$.
	Let $\mathcal E_n$ be the set of triples $(d,\pi,r)$ such that $d$ is a nonnegative integer not larger than $n$, $\pi \in \Pi(n,d)$ and $r$ is a positive integer.
	The linear order $<_{\mathcal E_n}$ on $\mathcal E_n$ is the lexicographic order.
	We abbreviate the subscript $\mathcal E_n$ of $<_{\mathcal E_n}$ in the rest of the paper, but it will not confuse readers.
	Let $X$ be a nonempty bounded definable subset of $F^n$.
	The triple $(d,\pi,r)$ is the \textit{extended rank} of $X$ and denoted by  $\myedim_n(X)$ if it is an element of $\mathcal E_n$ satisfying the following conditions:
	\begin{itemize}
		\item $d = \dim X$;
		\item the projection $\pi$ is a largest element in $\Pi(n,d)$ such that $\pi(X)$ has a nonempty interior;
		\item  $r$ is a largest positive integer such that there exists a definable open subset $U$ of $F^d$ contained in $\pi(X)$ such that the set $\pi^{-1}(x) \cap X$ is of dimension zero and the equality $\myrank(\pi^{-1}(x) \cap X) = r$ holds  for each $x \in U$.
	\end{itemize}
	Note that such a positive integer $r$ exists by \cite[Lemma 5.10]{Fornasiero}.
	We set $\myedim_n(\emptyset)=-\infty$ and define that $-\infty$ is smaller than any element in $\mathcal E_n$.
	
	Let us consider the case in which $X$ is an unbounded definable subset of $F^n$.
	Let $\varphi:F \to (-1,1)$ be the definable homeomorphism given by $\varphi(x)=\frac{x}{\sqrt{1+x^2}}$.
	We define $\varphi_n:F^n \to (-1,1)^n$ by $\varphi_n(x_1,\dots, x_n)=(\varphi(x_1),\ldots, \varphi(x_n))$.
	We set $\myedim_n(X)=\myedim_n(\varphi_n(X))$.
	
	Note that $\myedim$ obviously satisfies conditions (a) through (c) in Definition \ref{def:extended_rank}.
\end{definition}

\begin{lemma}\label{lem:inval_homeo}
	Consider a d-minimal expansion of an ordered field $\mathcal F=(F,<,+,\cdot,0,1,\ldots)$.
	Let $\varphi:F \to (-1,1)$ be the definable homeomorphism given in Definition \ref{def:extended_dim_dmin}.
	Let $X$ be a nonempty bounded definable subset of $F^n$.
	Then the equality $\myedim_n(X)=\myedim_n(\varphi_n(X))$ holds.
\end{lemma}
\begin{proof}
	Set $(d_1,\pi_1,r_1)=\myedim_n(X)$ and $(d_2,\pi_2,r_2)=\myedim_n(\varphi_n(X))$.
	It is obvious that $\pi(X)$ has a nonempty interior if and only if $\pi(\varphi_n(X))$ has a nonempty interior for each coordinate projection $\pi$.
	It implies that $d_1=d_2$ and $\pi_1=\pi_2$.
	For any $x \in \pi_1(X)$, we have $\mycl_{F^n}(\varphi_n(X) \cap \pi_1^{-1}(\varphi_d(x))))=\varphi_n(\mycl_{F^n}(X \cap \pi_1^{-1}(x)))$ because $X$ is bounded.
	Therefore, a point $y$ is isolated in $\mycl_{F^n}(X \cap \pi_1^{-1}(x))$ if and only if $\varphi_n(y)$ is isolated in $\mycl_{F^n}(\varphi_n(X) \cap \pi_1^{-1}(\varphi_d(x))))$.
	It implies that $r_1=r_2$.
\end{proof}

\begin{lemma}\label{lem:frontier_extended}
Consider a d-minimal expansion of an ordered field $\mathcal F=(F,<,+,\cdot,0,1,\ldots)$.
We have $\myedim_n \partial X < \myedim_n X$ for each nonempty definable subset $X$ of $F^n$.
\end{lemma}
\begin{proof}
	We may assume that $X$ is bounded without loss of generality by Lemma \ref{lem:inval_homeo}.
	We have nothing to prove when $\partial X$ is empty.
	We assume that $\partial X$ is not empty.
	Set $(d, \pi,r)=\myedim X$ and $(d',\pi',r')=\myedim \partial X$.
	We have $d' \leq d$ by Lemma \ref{lem:dim} and Lemma \ref{lem:dim2}.
	The lemma is obvious when $d'<d$.
	
	We consider the case in which $d'=d$.
	We show that $\pi'(X)$ has a nonempty interior.
	Set $Y=\{x \in F^d\;|\; \partial X \cap (\pi')^{-1}(x) \neq \partial (X \cap  (\pi')^{-1}(x))\}$.
	It has an empty interior by  Lemma \ref{lem:dim4}.
	On the other hand, $ \pi'(X)$ contains $\pi'(\partial X) \setminus Y$ by the definition of $Y$.
	They imply that $\pi'(X)$ has a nonempty interior by Lemma \ref{lem:dim2}.
	We have demonstrated that $\pi' \leq \pi$.
	
	The lemma follows when $\pi'<\pi$.
	Therefore, we consider the case in which $\pi'=\pi$.
	Set $Z_1=\{x \in F^d\;|\; \partial X \cap \pi^{-1}(x) \neq \partial (X \cap  \pi^{-1}(x))\}$ and $Z_2=\{x \in F^d\;|\; \partial X \cap \pi^{-1}(x) = \partial (X \cap  \pi^{-1}(x)) \neq \emptyset\}$.
	The definable set $Z_1$ has an empty interior by Lemma \ref{lem:dim4}.
	In addition, for any point $x \in Z_2$, we have $\myrank(\partial X \cap \pi^{-1}(x) ) < \myrank(X \cap \pi^{-1}(x)) $ by Corollary \ref{cor:rank_frontier}.
	They imply that $r'<r$.
\end{proof}

We recall the definition of $\Pi$-good sets defined in \cite{Fornasiero}.
\begin{definition}
	Consider an expansion of dense linear order without endpoints $\mathcal F=(F,<,\ldots)$.
	Let $X$ be a definable subset of $F^n$ of dimension $d$, and $\pi:F^n \to F^d$ be a coordinate projection.
	We say that $X$ is \textit{$\pi$-good} if the following conditions are satisfied:
	\begin{itemize}
		\item $\pi(X)$ is open;
		\item $\pi(X \cap U)$ has a nonempty interior for each $x \in X$ and an open neighborhood $U$ of $x$;
		\item For all $x \in \pi(X)$, the equalities $\dim (X \cap \pi^{-1}(x))=0$ and $\mycl(X \cap \pi^{-1}(x))=\mycl(X) \cap \pi^{-1}(x)$ hold.
	\end{itemize}
	We say that $X$ is \textit{$\Pi$-good} if $X$ is $\pi$-good for some coordinate projection $\pi:F^n \to F^d$.
	A collection of definable sets is called \textit{$\Pi$-good} if each of its elements is.
\end{definition}

\begin{proposition}\label{prop:deomp_into_good}
Consider a d-minimal expansion of an ordered field $\mathcal F=(F,<,+,\cdot,0,1,\ldots)$.
Let $\mathcal A$ be a finite collection of definable subsets of $F^n$.
Then, there exists a finite partition $\mathcal B$ of $F^n$ into $\Pi$-good definable subsets of $F^n$ such that every element of $\mathcal A$ is a union of sets in $\mathcal B$.

Furthermore, we can choose $\mathcal B$ so that the frontier of each element of $\mathcal B$ is a union of sets in $\mathcal B$.
\end{proposition}
\begin{proof}
	The proposition other than `furthermore' part is given in \cite[Lemma 3.17]{Fornasiero}  under more general assumption.
	We have only to prove the `furthermore' part.
	
	Let $(d,\pi,r)\in \mathcal E_n$.
	By reverse induction on $(d,\pi,r)$, we construct a decomposition $\mathcal B_{(d,\pi,r)}$ of $F^n$ into $\Pi$-good definable sets partitioning $\mathcal A$ such that the closures of all the $\Pi$-good definable sets in $\mathcal B_{(d,\pi,r)}$ of extended rank not smaller than $(d,\pi,r)$ are the unions of subfamilies of the decomposition. 
	Reverse induction is possible because $\myedim$ satisfies condition (c) in Definition \ref{def:extended_rank}.
	
	Let us consider the case in which $(d,\pi,r)=(n,\operatorname{id},1)$.
	The element $(n,\operatorname{id},1)$ is the largest element in $\mathcal E_n$.
	Here, $\operatorname{id}$ is the identity map on $F^n$.
	Apply the first part of the proposition to the family $\mathcal A$. 
	We obtain a finite partition $\mathcal B$ of $F^n$ into definable $\Pi$-good sets partitioning $\mathcal A$.
	Set $\mathcal C=\{B \in \mathcal B\;|\; \dim B=n\}$.
	Apply the first part of the proposition to the family $\mathcal A \cup \{\partial C\}_{C \in \mathcal C}$. 
	We get the decomposition $\mathcal B'$.
	Consider the set 
	$$
	\mathcal D=\{B' \in \mathcal B' \;|\; B' \text{ is not contained in any element }C \text{ in }\mathcal C\}\text{.}
	$$
	It is obvious that, for each $D \in \mathcal D$,  $\myrank D < (n,\operatorname{id},1)$ by Lemma \ref{lem:frontier_extended}.
	The partition $\mathcal C \cup \mathcal D$ is a partition of $F^n$ we are looking for.

	We next consider the case in which $(d,\pi,r)<(n,\operatorname{id},1)$.
	Let $\mathcal B$ be a decomposition of $F^n$ into $\Pi$-good definable sets partitioning $\mathcal A$ such that the closures of all the $\Pi$-good sets in $\mathcal B$ of extended rank greater than $(d,\pi,r)$ are the unions of subfamilies of the decomposition. 
	It exists by the induction hypothesis.
	Set $\mathcal C=\{B \in \mathcal B\;|\; \myedim B \geq (d,\pi,r)\}$ and $\mathcal C'=\{B \in \mathcal B\;|\; \myedim B = (d,\pi,r)\}$.
	Apply the first part of the proposition to the family $\mathcal A \cup \{\partial C\}_{C \in \mathcal C'}$. 
	We get the decomposition $\mathcal B'$.
	Consider the set 
	$$
	\mathcal D=\{B' \in \mathcal B' \;|\; B' \text{ is not contained in any element }C \text{ in }\mathcal C\}\text{.}
	$$
	It is obvious that, for each $D \in \mathcal D$,  $\myrank D < (d,\pi,r)$ by Lemma \ref{lem:frontier_extended}.
	The partition $\mathcal C \cup \mathcal D$ is a partition of $F^n$ we want to construct.
\end{proof}

\begin{proposition}\label{prop:extended_dim_dmin}
	Consider a d-minimal expansion of an ordered field $\mathcal F=(F,<,+,\cdot,0,1,\ldots)$.
	The extended rank function $\myedim_n$ defined in Definition \ref{def:extended_dim_dmin} possesses properties (a) through (f) in Definition \ref{def:extended_rank}.
\end{proposition}
\begin{proof}
	Conditions (a) through (c) are obviously satisfied by the definition.
	Satisfaction of condition (d) follows from Lemma \ref{lem:dim2} and Lemma \ref{lem:very_basic3}.
	The triple $(0,\pi_0,1)$ is a smallest element of $\mathcal E_n$, where $\pi_0$ is the unique element of $\Pi(n,0)$.
	By the definition of rank, $X$ is discrete if $\myedim_n(X)=(0,\pi_0,1)$.
	It implies that $\myedim_n$ meets condition (f).
	
	The remaining task is to show that condition (e) is satisfied.
	Let $X, Y \in \operatorname{Def}(F^n)$ with $X \subseteq Y$.
	Assume that $\myedim_n(X)=\myedim_n(Y)$.
	We show that there exists a definable subset $T$ of $X$ such that $T$ is open in $Y$ and $\myedim_n(X \setminus T)<\myedim_n(Y)$. 
	
	We may assume that $Y$ is bounded by considering $\varphi_n(Y)$ in place of $Y$ if necessary by Lemma \ref{lem:inval_homeo}.
	Let $(d,\pi,r)=\myedim_n(Y)$.
	Since $\pi(Y)$ has a nonempty interior, there exists a definable dense open subset $U_Y$ such that $\pi^{-1}(x) \cap Y$ is of dimension zero for each $x \in U_Y$ by Lemma \ref{lem:dim3}.
	By the definition of extended rank, there exist a definable open subset $U_X$ of $F^d$ contained in $\pi(X)$ such that  the set $\pi^{-1}(x) \cap X$ is of dimension zero and the inequality $\myrank(\pi^{-1}(x) \cap X) = r$ holds  for each $x \in U_X$.
	We set $U=U_X \cap U_Y$.
	It is a definable open subset of $F^d$ contained in $\pi(X)$ such that $\pi^{-1}(x) \cap Y$ is of dimension zero and the equalities $\myrank(\pi^{-1}(x) \cap X) = r$ and $\myrank(\pi^{-1}(x) \cap Y) = r$ hold  for each $x \in U$.
	
	Apply Proposition \ref{prop:deomp_into_good} to the family $\mathcal A=\{X,Y,\pi^{-1}(U)\}$.
	There is a partition $Y=\bigcup_{i=1}^mQ_i$ of $Y$ into $\Pi$-good definable sets such that  $X$, $\pi^{-1}(U)$ and $\mycl(Q_i) \setminus Q_i$ are finite unions of elements in  $\mathcal Q:= \{Q_1, \ldots, Q_m\}$ for all $1 \leq i \leq m$.
	By renumbering elements of $\mathcal Q$ if necessary, we may assume that $X=\bigcup_{i=1}^l Q_i$.
	We have $\myedim_n(Q_i)=(d,\pi,r)$ for some $1 \leq i \leq l$ by property (d) in Definition \ref{def:extended_rank}, which is already proven.
	We may assume that $\myedim_n(Q_i)=(d,\pi,r)$ for $1 \leq i \leq k$ and $\myedim_n(Q_i)<(d,\pi,r)$ for $k<i \leq l$ by renumbering them again.
	
	We prove that $Q_i$ are open in $Y$ for each $1 \leq i \leq k$.
	Assume for contradiction that $Q_i \cap \mycl_Y(Y \setminus Q_i) \neq \emptyset$.
	There exists $1 \leq j \leq m$ such that $j \neq i$ and $Q_i \cap \mycl_Y(Q_j) \neq \emptyset$ because $Y=\bigcup_{i=1}^mQ_i$.
	Therefore, $Q_i$ is contained in $\mycl_Y(Q_j) \setminus Q_j$ by the definition of $\mathcal Q$.
	Since $\mycl_Y(Q_j)$ is a finite union of elements in $\mathcal Q$ and contains $Q_i$, we have $\myedim(\mycl_Y(Q_j))=(d,\pi,r)$ by property (d) in Definition \ref{def:extended_rank}.
	There exists a definable dense open subset $V$ of $F^d$ such that $\mycl(Q_j \cap \pi^{-1}(x))=\mycl(Q_j) \cap \pi^{-1}(x)$ for each $x \in V$ by Lemma \ref{lem:dim3}.
	By the definition of extended rank, there exists a definable open set $W$ of $F^d$ contained in $\pi(Q_i)$ such that $\myrank(\pi^{-1}(x) \cap Q_i) =r$ for each $x \in W$.
	For any point $x \in V \cap W$, we obtain $Q_i \cap \pi^{-1}(x) \subseteq (\mycl(Q_j) \setminus Q_j) \cap \pi^{-1}(x) = \mycl(Q_j\cap \pi^{-1}(x)) \setminus (Q_j \cap \pi^{-1}(x))$.
	Note that each isolated point in $\mycl(Q_j \cap \pi^{-1}(x))$ belongs to $Q_j \cap \pi^{-1}(x)$.
	They imply that $\myrank(Q_j \cap \pi^{-1}(x))>r$ for each $x \in V \cap W$.
	It is a contradiction to the inequality $\myedim_n(Q_j) \leq (d,\pi,r)$.
	We have proven that $Q_i$ is open in $Y$ for each $1 \leq i \leq k$.
	
	Set $T=\bigcup_{i=1}^k Q_i$.
	It is open in $Y$ because $Q_i$ is open in $Y$ for each $1 \leq i \leq k$.
	We have $\myedim(X \setminus T)=\myedim_n(\bigcup_{i=k+1}^lQ_i) < (d,\pi,r)$ by property (d) in Definition \ref{def:extended_rank}.
\end{proof}

We next prove that a d-minimal structure admits a good $(0+)$-pseudo-filter.
\begin{definition}
	Consider a d-minimal expansion of an ordered field $\mathcal F=(F,<,+,\cdot,0,1,\ldots)$.
	A definable subset $D$ of $F$ is a \textit{$(0+)$-set} if it satisfies the following:
	\begin{itemize}
		\item It is bounded and discrete in $F$;
		\item It is contained in $\{r \in F\;|\; r>0\}$;
		\item $\mycl_{F}(D)=D \cup \{0\}$.
	\end{itemize}
\end{definition}

\begin{proposition}\label{prop:filter_gmin}
	Consider a d-minimal expansion of an ordered field $\mathcal F=(F,<,+,\cdot,0,1,\ldots)$ which is not locally o-minimal.
	The family $\mathfrak D$ of $(0+)$-sets is not empty and enjoys properties (a) through (c) in Definition \ref{def:filter}. 
\end{proposition}
\begin{proof}
	We first prove that $\mathfrak D$ is not an empty set.
	Any definable set is constructible by Lemma \ref{lem:constructible} and Proposition \ref{prop:extended_dim_dmin}.
	There exists an unbounded discrete definable subset $X$ of $F$ by \cite[Theorem 3.3]{Fornasiero0} together with the fact every definable set is constructible.
	We may assume that $x \geq 1 $ for any $x \in X$ without loss of generality.
	The set $\{x \in F\;|\; 1/x \in X\}$ belongs to $\mathcal D$. 
	
	We next check that $\mathfrak D$ enjoys property (a) through (c) in Definition \ref{def:filter}.
	The family $\mathfrak D$ obviously possesses property (a) by the definition of $(0+)$-set.
	Let $D \in \mathfrak D$.
	We prove the following claim:
	\medskip
	
	\textbf{Claim.} A  definable subset $D'$ of $D$ is an element of $\mathfrak D$ when $D \in \mathfrak D$ and $D'$ is not closed in $F$.
	\begin{proof}[Proof of Claim]
		The set $D'$ is discrete because $D$ is discrete.
		We have $\mycl_{F}(D')=D' \cup \{0\}$ because $D'$ is not closed and  each point in $D$ is isolated in $D$. 
	\end{proof}
	Claim immediately implies property (b)  in Definition \ref{def:filter}.
	
	The final task is to demonstrate that $\mathfrak D$ enjoys property (c).
	Let $D \in \mathfrak D$ and $\gamma:D \to F^n$ be a definable map.
	Any map defined on a discrete set is continuous.
	It implies that $\gamma$ is continuous.
	Therefore, we have only to show that there exists $D' \in \mathfrak D$ such that $D' \subseteq D$ and $\myLim(\gamma|_{D'})$ is a singleton when $\myLim(\gamma)$ is not empty.
	
	Note that $\gamma(D)$ is of dimension zero by Lemma \ref{lem:dimzeo}.
	Since $\myLim(\gamma)$ is a subset of the closure $\mycl(\gamma(D))$, it is at most of dimension zero by Lemma \ref{lem:dim} and Lemma \ref{lem:dim2}.
	Therefore, $\myLim(\gamma)$ has an isolated point, say $y$, by Lemma \ref{lem:dmin_countable}.
	We first consider the case in which $E=\{t \in D\;|\;\gamma(t)=y\}$ is not closed.
	We have $E \in \mathfrak D$ by Claim.
	In this case, we have only to set $D'=E$.
	
	We next consider the other case.
	Since $E$ is bounded and closed, we have $D \setminus E \in \mathfrak D$ and $\myLim(\gamma|_{D \setminus E})=\myLim(\gamma)$.
	We may assume that $\gamma(t) \neq y$ for all $t \in D$ by replacing $D$ with $D \setminus E$.
	We can take an open box $V$ such that $\mycl_{F^n}(V) \cap \myLim(\gamma)=\{y\}$ because $y$ is an isolated point in $\myLim(\gamma)$.
	Set $D'=\gamma^{-1}(V)$.
	It is not a closed set.
	Assume for contradiction that $D'$ is closed.
	The image $\gamma(D')=V \cap \gamma(D)$ is closed by Lemma \ref{lem:image} because $\gamma$ is continuous and $D'$ is closed and bounded.
	However, $y \in \partial (V \cap \gamma(D))$ by the assumption.
	It is a contradiction as desired.
	Consequently, we have proven that $D'$ is a $(0+)$-set by Claim.
	Let $\gamma'$ be the restriction of $\gamma$ to $D'$.
	It is obvious that $y \in \myLim(\gamma')$.
	On the other hand, we have $\myLim(\gamma') \subseteq \myLim(\gamma) \cap \mycl(V)=\{y\}$.
	We have shown that $\myLim(\gamma')=\{y\}$. 
\end{proof}

\begin{theorem}\label{thm:summary}
	A d-minimal expansion of an ordered field satisfies the assumptions in Theorem \ref{thm:quotinet}; that is, it enjoys the definable choice property,  admits a good $(0+)$-pseudo-filter $\mathfrak D$, and has a good extended rank function $\myedim$.
\end{theorem}
\begin{proof}
	When the structure is locally o-minimal, the theorem follows from Remark \ref{rem:locally_omin}.
	When the structure is not locally o-minimal, the theorem follows from \cite{Miller-choice}, Proposition \ref{prop:extended_dim_dmin} and Proposition \ref{prop:filter_gmin}.
\end{proof}

\appendix
\section{Definable Tietze extension theorem for bounded functions}\label{sec:extension_thm}
We need definable Tietze extension theorem for the proof of existence of definable quotient.
We prove it in this section.

\begin{definition}
	Let $\mathcal F=(F,<,\ldots)$ be an expansion of a dense linear order without endpoints.
	A pair $(X,d)$ of a definable set $X$ and a definable function on $X$ is a \textit{definable metric space} if $d$ is a metric on $X$. 
\end{definition}

\begin{theorem}[Definable Tietze extension theorem for bounded definable functions]\label{thm:tietze_abstract}
	Let $\mathcal F=(F,<,+,\cdot,0,1,\ldots)$ be a definably complete expansion of an ordered field.
	Let $(X,d)$ be a definable metric space.
	We consider the topology induced from the metric $d$.
	Let $A$ be a nonempty definable closed subset of $X$ and $\varphi:A \to F$ be a bounded definable continuous function.
	Then there exists a definable continuous extension $\Phi:X \to F$ of $\varphi$.
\end{theorem}
\begin{proof}
	The proof is almost the same as \cite[Lemma 6.6]{AF}.
	However, the proof of \cite[Lemma 6.6]{AF} only gives how to construct $\Phi$, but does not prove that the given $\Phi$ is continuous.
	We give a proof here for readers' convenience.
	
	Set $B=(1,2)$.
	We first reduce to the case in which $\varphi(A) \subseteq [R_1,R_2]$ for some $1<R_1<R_2<2$. 
	Since $\varphi$ is bounded, there exists $S_1<S_2$ such that $\varphi(A) \subseteq [S_1,S_2]$.
	Let $\tau:F \to B$ be the definable homeomorphism defined by $$\tau(t)=\dfrac{3}{2}+\dfrac{t}{2\sqrt{1+t^2}}.$$
	Put $R_i=\tau(S_i)$ for $i=1,2$.
	By the assumption, we have $\tau \circ \varphi(A) \subseteq [R_1,R_2]$.
	If we succeed in constructing a definable continuous extension $\Psi:X \to F$ of $\tau \circ \varphi$ such that $\Psi(X) \subseteq [R_1,R_2]$, the map $\Phi:=\tau^{-1} \circ \Psi$ is a definable continuous extension of $\varphi$.
	Therefore, we may assume that $\varphi(A) \subseteq [R_1,R_2]$ by considering $\tau \circ \varphi$ instead of $\varphi$.
	We construct a definable continuous extension $\Phi:X \to F$ of $\varphi$ such that $\Phi(X) \subseteq [R_1,R_2]$.
	
	We define the function $\Phi:X \to F$ by $\Phi(x)=\varphi(x)$ when $x \in A$ and $$\Phi(x)=\inf_{a \in A}\varphi(a)\cdot\dfrac{d(x,a)}{d(x,A)}$$ elsewhere.
	Here, $d(x,A)=\inf\{d(x,a)\;|\;a \in A\}$.
	This function is definable and well-defined because $\mathcal F$ is definably complete.
	It is obvious the restriction of $\Phi$ to $A$ is $\varphi$.
	The remaining task is to show that $\Phi$ is continuous and $\Phi(X) \subseteq [R_1,R_2]$.
	
	We first prove the following claim:
	\medskip
	
	\textbf{Claim 1.} Let $x \in X \setminus A$.
	Set $A_x:=\{a \in A\;|\; d(x,a) \leq 2d(x,A)\}$.
	Then we have $d(x,A)=d(x,A_x)$ and $$\Phi(x)=\inf_{a \in A_x}\varphi(a)\cdot\dfrac{d(x,a)}{d(x,A)}.$$
	In addition, we have $R_1 \leq \Phi(x) \leq R_2$ for each $x \in X$.
	\begin{proof}[Proof of Claim 1]
		It is obvious that $d(x,A)=d(x,A_x)$ by the definition of $A_x$.
		It is also obvious that $\Phi(x) \geq R_1$ because $\varphi(a)\cdot\dfrac{d(x,a)}{d(x,A)} > R_1\cdot \dfrac{d(x,a)}{d(x,A)} \geq R_1$ for each $a \in A$.
		For any $\varepsilon >0$, there exists $a' \in A$ such that $0 \leq d(x,a')-d(x,A) \leq \varepsilon d(x,A)$ by the definition of $d(x,A)$.
		This inequality implies that $\Phi(x) \leq \varphi(a')\cdot\dfrac{d(x,a')}{d(x,A_x)} \leq R_2+R_2\varepsilon$ because $\varphi(a') \leq R_2$.
		Since $\varepsilon$ is arbitrary, we get $\Phi(x) \leq R_2<2$.
		
		On the other hand, we have $\varphi(a)\cdot\dfrac{d(x,a)}{d(x,A)}>2$ for $a \in A \setminus A_x$ because of the definition of $A_x$ and the assumption that $\varphi(a)>1$.
		It implies that $\inf_{a \in A}\varphi(a)\cdot\dfrac{d(x,a)}{d(x,A_x)}=\inf_{a \in A_x}\varphi(a)\cdot\dfrac{d(x,a)}{d(x,A_x)}$.
	\end{proof}

	We use Claim 1 without notice in the proof.
	
	We fix an arbitrary point $x_1 \in X$.
	We prove that $\Phi$ is continuous at $x_1$.
	We first consider the case in which $x_1 \in A$.
	We fix $\varepsilon>0$.
	Since $\varphi$ is continuous, there exists $\delta>0$ such that, for any $a \in A$ with $d(a,x_1)<3\delta$, the inequality $|\varphi(a)-\varphi(x_1)|<\varepsilon/4$ holds.
	We want to prove that $|\Phi(x_2)-\Phi(x_1)|=|\Phi(x_2)-\varphi(x_1)|<\varepsilon$ for each $x_2 \in X$ with $d(x_1,x_2)<\delta$.
	This inequality is obvious when $x_2 \in A$.
	Therefore, we only consider the case in which $x_2 \in X \setminus A$.
	We put $A_2=A_{x_2}$.
	
	For each $a \in A_2$, we have $d(a,x_1) \leq d(a,x_2)+d(x_1,x_2) \leq 2d(x_2,A)+d(x_1,x_2) \leq 2d(x_1,x_2)+d(x_1,x_2)<3\delta$ because $x_1 \in A$.
	By the definition of $\delta$, we have $\varphi(a)>\varphi(x_1)-\varepsilon/4$.
	We also have $d(x_2,a) \geq d(x_2,A)$ for each $a \in A$.
	These inequality imply that $\varphi(a)\cdot\dfrac{d(x_2,a)}{d(x_2,A)}>\varphi(x_1)-\varepsilon/4$ for each $a \in  A_2$.
	It implies that $$\Phi(x_2) \geq \varphi(x_1)-\varepsilon/4.$$
	
	We next take $\varepsilon'>0$ so that $\varepsilon'<\min\{1,\varepsilon/4\}$.
	We can take $a' \in A$ so that $d(x_2,a') \leq (1+\varepsilon')d(x_2,A)$ because $x_2 \not\in A$.
	We have $d(x_2,a')<2d(x_2,A)$ because $\varepsilon'<1$.
	It means that $a' \in A_2$.
	Therefore, we have $d(a',x_1)<3\delta$.
	By the definition of $\varepsilon$, we have $\varphi(a')<\varphi(x_1)+\varepsilon'/4$.
	We get 
	\begin{align*}
		\Phi(x_2) &\leq  \varphi(a')\cdot\dfrac{d(x,a')}{d(x,A)} <(\varphi(x_1)+\varepsilon/4))\cdot (1+\varepsilon')\\
		&=\varphi(x_1)+\varphi(x_1)\varepsilon'+\varepsilon/4+\varepsilon\varepsilon'/4\\
		&<\varphi(x_1)+2\varepsilon'+\varepsilon/4+\varepsilon/4\\
		&<\varphi(x_1)+\varepsilon.
	\end{align*}
	We have proven that $|\Phi(x_2) -\varphi(x_1)|<\varepsilon$.
	It means that $\Phi$ is continuous at $x_1$.
	
	We next consider the case in which $x_1 \notin A$.
	We prove the following claim:
	\medskip
	
	\textbf{Claim 2.} For any $\varepsilon>0$, there exists $\delta>0$ such that, if $x_2 \in X$ and $d(x_1,x_2)<\delta$, we have $x_2 \in X \setminus A$ and $$\left|\varphi(a)\cdot\dfrac{d(x_2,a)}{d(x_2,A)}-\varphi(a)\cdot\dfrac{d(x_1,a)}{d(x_1,A)}\right| <\dfrac{\varepsilon}{2}$$ for each $a \in A_1 \cup A_2$.
	Here, we set $A_i=A_{x_i}$ for each $i=1,2$.
	
	\begin{proof}[Proof of Claim 2]
	Let $\varepsilon>0$ be an arbitrary positive element in $F$.
	We take $\delta>0$ satisfying the following conditions:
	\begin{itemize}
		\item $0<\delta<d(x_1,A)$;
		\item $\dfrac{\delta}{d(x_1,A)-\delta}<\dfrac{\varepsilon}{8}$;
		\item The inequality $\left|\dfrac{1}{d(x_1,A)}-\dfrac{1}{d(x_2,A)}\right|<\dfrac{\varepsilon}{40 \cdot d(x_1,A)}$ holds for each $x_2 \in X$ with $d(x_1,x_2)<\delta$.
	\end{itemize}
	It is possible because the maps defined by $(-\infty,d(x_1,A)) \ni x \mapsto x/(d(x_1,A)-x) \in F$ and $X \setminus A \ni x \mapsto 1/d(x,A) \in F$ are continuous at $0$ and $x_1$, respectively.
	
	We fix $x_2 \in X$ with $d(x_1,x_2) < \delta$.
	It is obvious that $x_2 \in X \setminus A$ because $d(x_1,x_2)<d(x_1,A)$.
	We prove $d(x_1,a) < 5d(x_1,A)$ for each $a \in A_1 \cup A_2$.
	When $a \in A_1$, the definition of $A_1$ implies that $d(x_1,a) \leq 2d(x_1,A) < 5 d(x_1,A)$.
	When $a \in A_2$, we have 
	\begin{align*}
		d(x_1,a) &\leq d(x_1,x_2)+d(x_2,a) \leq d(x_1,x_2) +2d(x_2,A) \\
		&\leq d(x_1,x_2)+2(d(x_1,A)+d(x_1,x_2)) <2d(x_1,A)+3\delta\\
		&<5d(x_1,A).
	\end{align*}
	We have proven the inequality $d(x_1,a) < 5d(x_1,A)$ for each $a \in A_1 \cup A_2$.
	We also have $d(x_2,A) \geq d(x_1,A) - d(x_1,x_2) > d(x_1,A) -\delta$ and $|d(x_2,a)-d(x_1,a)| \leq d(x_1,x_2)<\delta$.
	Using these inequalities, for each $a \in A_1 \cup A_2$, we get
	\begin{align*}
		&\left|\varphi(a)\cdot\dfrac{d(x_2,a)}{d(x_2,A)}-\varphi(a)\cdot\dfrac{d(x_1,a)}{d(x_1,A)}\right| \leq |\varphi(a)| \cdot \left|\dfrac{d(x_2,a)}{d(x_2,A)}- \dfrac{d(x_1,a)}{d(x_1,A)}\right|\\
		& < 2 \left(\left|\dfrac{d(x_2,a)}{d(x_2,A)}- \dfrac{d(x_1,a)}{d(x_2,A)}\right|+\left|\dfrac{d(x_1,a)}{d(x_2,A)}- \dfrac{d(x_1,a)}{d(x_1,A)}\right|\right)\\
		&=2\left(\dfrac{|d(x_2,a)-d(x_1,a)|}{d(x_2,A)}+d(x_1,a)\left|\dfrac{1}{d(x_2,A)}- \dfrac{1}{d(x_1,A)}\right|\right)\\
		&<2\left(\dfrac{\delta}{d(x_1,A)-\delta}+5d(x_1,A) \cdot \dfrac{\varepsilon}{40 \cdot d(x_1,A)}\right)<\varepsilon/2.
	\end{align*}
We have proven Claim 2.
\end{proof}
	The preparation has been done. 
	We lead to a contradiction assuming that $\Phi$ is not continuous at $x_1$.
	There exists $\varepsilon>0$ such that, for any $\delta>0$, there exists $x _2 \in X$ satisfying the inequalities $d(x_1,x_2)<\delta$ and $|\Phi(x_1)-\Phi(x_2)| \geq \varepsilon$.
	We take $\delta>0$ so that the inequality in Claim 2 holds for each $x_2 \in X$ with $d(x_1,x_2)<\delta$.
	We consider two separate cases.
	When $\Phi(x_1)<\Phi(x_2)$, we can take $a \in A_1$ so that $$\Phi(x_1)>\varphi(a)\cdot\dfrac{d(x_1,a)}{d(x_1,A)}-\dfrac{\varepsilon}{4}.$$
	Applying the inequality in Claim 2, we get $$\Phi(x_1) > \varphi(a)\cdot\dfrac{d(x_2,a)}{d(x_2,A)}-\dfrac{3\varepsilon}{4} \geq \Phi(x_2)- 3\varepsilon/4.$$
	It implies that $\Phi(x_2)- 3\varepsilon/4 < \Phi(x_1) < \Phi(x_2)$.
	It contradicts the assumption that $|\Phi(x_1)-\Phi(x_2)| \geq \varepsilon$.
	We can treat the case in which $\Phi(x_1)>\Phi(x_2)$ similarly.
\end{proof}

\begin{remark}
	We make several comments on Theorem \ref{thm:tietze_abstract}.
	The image $\Phi(X)$ of $X$ under the function $\Phi$ constructed in the proof is not necessarily contained in $B$ when $\varphi(A)=B$.
	In fact, let $d(x,y)=\sqrt{(x_1-y_1)^2+(x_2-y_2)^2}$ for $x=(x_1,x_2), y=(y_1,y_2) \in \mathbb R^2$. 
	Let $D$ and $S$ be the closed disc and circle of radius $1/2$ centered at $(3/2,0)$, respectively.
	Set $X=D \setminus \{(1,0),(2,0)\}$ and $A=S \setminus \{(1,0),(2,0)\}$ and $\varphi(x_1,x_2)=x_1$ for each $(x_1,x_2) \in X$.
	We have $\Phi(3/2,0)=1 \notin B$ in this case.
	
	We can drop the assumption that $\varphi$ is bounded in the following cases:
	\begin{enumerate}
		\item[(1)] $X=F^n$;
		\item[(2)] $\mathcal F$ is o-minimal.
	\end{enumerate}
	In fact, the case (1) is \cite[Lemma 6.6]{AF}.
	We can find $a,a' \in A$ such that $\Phi(x)=\varphi(a)\cdot\dfrac{d(x,a)}{d(x,A)}$ and $d(x,a')=d(x,A)$ for each $x \in X \setminus A$ using Lemma \ref{lem:image} and the equality given in Claim 1.
	The inclusion $\Psi(X) \subseteq B$ immediately follows from these facts.
	The remaining proof is the same as  that of Theorem \ref{thm:tietze_abstract}.
	The case (2) is treated in \cite[Chapter 8, Corollary 3.10]{D}.
	It is proven using triangulation of a definable set.
\end{remark}

\end{document}